\documentclass[11pt, a4paper]{article}

\usepackage[utf8]{inputenc}
\title{Hydrodynamics Formulation for Deep Learning and Lie--Poisson Hamilton--Jacobi Theory}
\author{Nader T. Ganaba \thanks{\href{mailto:nganaba@outlook.com}{\texttt{nganaba@outlook.com}}}}

\usepackage{geometry}
\usepackage{tikz}
\usepackage{cite}

\usepackage[shortlabels]{enumitem}
\usepackage{amsmath,amsfonts,amssymb,amsthm}
\usepackage{comment}
\usepackage[font=scriptsize,labelfont=bf]{caption}
\usepackage{subcaption}
\usepackage[hidelinks]{hyperref}
\usepackage{tikz-cd} 
\newtheorem{theorem}{Theorem}[section]
\newtheorem{corollary}[theorem]{Corollary}
\newtheorem{proposition}[theorem]{Proposition}

\newtheorem{definition}[theorem]{Definition}

\newtheorem{problem}[theorem]{Problem}
\newtheorem{remark}[theorem]{Remark}

\def\contrac{\mathbin{\hbox to 6pt{%
                 \vrule height0.4pt width5pt depth0pt
                 \kern-.4pt
                 \vrule height6pt width0.4pt depth0pt\hss}}}

\begin{document}

\maketitle

\begin{abstract}
The interpretation of deep learning as a dynamical system has gained a considerable attention in recent years as it provides a promising framework. It allows for the use of existing ideas from established fields of mathematics for studying deep neural networks. In this article we present deep learning as an equivalent hydrodynamics formulation which in fact possess a Lie-Poisson structure and this further allows for using Poisson geometry to study deep learning. This is possible by considering the training of a deep neural network as a stochastic optimal control problem, which is solved using mean--field type control. The optimality conditions for the stochastic optimal control problem yield a system of partial differential equations, which we reduce to a system of equations of quantum hydrodynamics.  We further take the hydrodynamics equivalence to show that, with conditions imposed on the network, that it is equivalent to the nonlinear Schr\"{o}dinger equation.

As the such systems have a particular geometric structure, we also present a structure--preserving numerical scheme based on the Hamilton--Jacobi theory and its Lie--Poisson version known as Lie--Poisson Hamilton--Jacobi theory. The choice of this scheme was decided by the fact it can be applied to both deterministic and mean--field type control Hamiltonians. Not only this method eliminates the possibility of the emergence of fictitious solutions, it also eliminates the need for additional constraints when constructing high order methods. 
\end{abstract}

\section{Introduction}
The kinship between deep learning and dynamical systems, discussed thoroughly and lucidly in \cite{weinan2017proposal}, has gained a considerable attention in recent years as it was studied in a number of papers such as \cite{li2017maximum, pmlr-v80-li18b, benning2019deep}. This comes as a consequence of the increased complexity of deep learning and their success in various real world applications, such as performing outstandingly in recognition tasks for images \cite{farabet_learning_2013,szegedy_going_2014, krizhevsky_imagenet_2017}, and speech \cite{hinton_deep_2012-1, seide_conversational_nodate}, automated language translation \cite{10.5555/2969033.2969173}, drug discovery \cite{ma_deep_2015, gawehn_deep_2016, chen_rise_2018} and predicting splicing patterns of RNA and their role in genetic diseases \cite{barash_deciphering_2010, leung2014deep, xiong2015human}. Yet, the inner workings are not well understood and the lack of mathematical framework in which one can analyse the systems. Further, basic tools of linear algebra are not sufficient to explain and analyse the behaviour of complex systems. 
\\

This article is part of a series that view deep learning as a hydrodynamics system within the geometric mechanics framework to better understand the behaviour of deep networks. And here, we introduce a Hamiltonian hydrodynamic system on Poisson manifold which is equivalent to training of deep neural networks. The result is a system of partial differential equations similar to quantum Euler equation, or Madelung equation and this equation can be further transformed into the nonlinear Schr\"{o}dinger equation. As these equations have special properties, we will also present geometric numerical schemes to be used as training algorithms. 
\\

The passage from deep learning to hydrodynamics is not direct and at its core is optimal control interpretation of deep learning.  The use of optimal control method for deep learning is not new, in fact, it is used to train residual neural networks \cite{li2017maximum,weinan2019mean,pmlr-v80-li18b, benning2019deep}. In addition to that, it is a candidate for a framework for studying deep neural networks \cite{Haber_2017, benning2019deep}. Besides training this point of view is important as there is an association between optimal control and geometry \cite{sussmann_geometry_1999, bloch_optimal_1999, gay-balmaz_clebsch_2011} and this offers the opportunity to use methods and theories beyond linear and nonlinear algebra to study deep learning. This association is due in part to Pontryagin's Maximum Principle \cite{doi:10.1002/zamm.19630431023} for solving the optimal control problem, where it gives a Hamiltonian function whose associated vector field generates a flow that solves the training problem. In Hamiltonian mechanics, the solutions can be considered as coordinate transformations between two symplectic manifolds known as \emph{symplectic transformations}, or \emph{symplectomorphisms}. The underlying geometry allows one to make the transition into continuum mechanics \cite{bloch2000optimal, holm_eulers_2009} and such transition is reminiscent of \cite{benamou2000computational} use of fluid dynamics for solving optimal transport problems. Here, we consider a general case when the deep neural network is stochastic as a way to quantify uncertainties. That said, the training problem becomes a stochastic optimal control problem and there are different ways to solve this problem. The approach chosen here is mean--field type control \cite{bensoussan_mean_2013}, which casts the problem as a deterministic problem. What is advantageous is that the conditions for optimality can be reduced to  an equivalent hydrodynamics system, which is indeed a \emph{Lie--Poisson system}. The resultant system of partial differential equations has a structure as the one for ideal compressible fluid. 
\\

Such formulation for deep learning is novel, and it is presented in this article. Not only it helps in applying Hamilton--Jacobi for Poisson integrators, but also provides a framework for studying and analysing deep learning problems. The said Lie--Poisson system for deep learning presented here has the same structure as the ideal compressible fluids; it allows us to use theories of hydrodynamics to study neural networks. The question of stability can help choosing hyperparameters heuristically and methods such as energy--momentum \cite{Simo1991StabilityOR, SIMO199263}, and energy--Casimir both deterministic \cite{HOLM19851} and stochastic \cite{arnaudon_stochastic_2018} can be used to study the stability of Lie--Poisson systems. That said, the study of stability is reserved for a future article. Here the focus is on employing Hamilton--Jacobi theory and its variants for systems with symmetry for deriving structure--preserving integrators for Hamiltonian systems. 
\\

Since these systems are Hamiltonian, one needs to use numerical schemes that are symplectic or Poisson. There are a number of ways one can derive such integrators, such as the partitioned Runge--Kutta method \cite{abia_partitioned_1993}, or variational integrators \cite{marsden2001discrete} and here we focus on a type that uses generating functions. It relies on the \emph{Hamilton--Jacobi theory}, which is an alternative formulation of Hamiltonian mechanics in terms of wavefronts  and instead of solving for a single path, it generates a family of solutions.  At the heart of the Hamilton--Jacobi theory, there is the \emph{Hamilton--Jacobi equation}, as it appears in geometrical optics \cite{caratheodory_grundlagen_1937,1980719}, and the equation for modelling the wave front propagation and this is considered a passage from quantum mechanics to classical mechanics \cite{keller_corrected_1958}.  The solution of the Hamilton--Jacobi equation, known as \emph{generating function}, which from the point of view of symplectic geometry, allows to heuristically construct coordinate transformations that preserve a geometric structure known as \emph{symplectic structure}. This forms a basis for the numerical scheme, as obtaining an approximation of the generating function we can then use it to generate symplectic transformations. In other words, the numerical scheme is a symplectic transformation. 
\\

Here, we limit the application of the Hamilton--Jacobi theory to the construction of a structure--preserving numerical scheme, based on the generating function, by projecting a continuous neural network model onto a finite dimensional space while conserves a number of the network's properties. The second application is more elaborate as it takes into account uncertainties in the network and uses mean--field games interpretation of deep learning and it uses a variant of Hamilton--Jacobi theory known as \emph{Lie--Poisson Hamilton--Jacobi theory} \cite{zhong1988lie}.  As mentioned, the resultant system of partial differential equations has the same Poisson structure as the ideal compressible fluid \cite{marsden_semidirect_1984}, and because it is defined on semidirect product spaces, it necessitates the derivation of Lie--Poisson Hamilton--Jacobi theory of semidirect products spaces.  Alternatively, the mean--field games equations ca be solved using \emph{multisymplectic variational integrators} \cite{marsden1998multisymplectic} and it is the focus of  \cite{G2021}. It is worth pointing out that such integrators were first applied to machine learning problems in \cite{barbaresco2020lie}. The difference between \cite{G2021} and \cite{barbaresco2020lie} is that the latter uses tools from information geometry, geometric statistical mechanics and Lie group thermodynamics for machine learning and data belong to a Lie group. Whereas the former, only the network's weights are elements of the Lie algebra and the multisymplectic partial differential equation arises from mean--field games control. 
\\

Most numerical methods for the class of Lie--Poisson equations for infinite dimensional systems use a discretisation of the vector field on the group of diffeomorphisms and not the group's elements  \cite{mclachlan_explicit_1993, engo_numerical_2001} based on the works of \cite{zeitlin_finite-mode_1991}. However, it is difficult to discretise the group of diffeomorphisms directly, as it is an infinite dimensional group, and apply Lie--Poisson Hamilton--Jacobi theory. One option is to consider the discretisation of \cite{pavlov2011structure} where the group of diffeomorphisms is approximated in finite dimensional space as the Lie group $GL(N)$. Such discretisation comes with its own challenges, such as the introduction of flat operators. For this reason we introduce a discretisation that uses the group of pseudodifferential symbols instead of the group of diffeomorphisms, which is a submanifold of the group of pseudodifferential symbols, and it makes systems on the infinite dimensional space compatible with the Lie--Poisson Hamilton--Jacobi theory. The reward of going to these lengths for numerical integrators is that the scheme preserves important properties such as the coadjoint orbit, meaning that the numerical solution would always belong to the solution set.  Figure \ref{fig:three_spheres} demonstrates visually the solution of the angular momentum of  Hamiltonian system of the $\mathrm{SO}(3)$ rigid body obtained by  three different numerical schemes. For this rigid body, the coadjoint orbit is a sphere. It is evident the explicit Euler and $4^{th}$ order Runge--Kutta method yield solutions that do not belong to the set of solutions, whereas the Lie--Poisson Hamilton--Jacobi scheme yields a solution that remains on the surface of the sphere. 
\\

The three main reasons for using generating functions to derive numerical schemes are high order integrators, possible lack of hyperregularity and its extension to the stochastic case. The first reason concerns the construction of high-order schemes and using generating functions one can construct high-order methods without the need to resort to carefully select coefficients, such as symplectic Runge--Kutta method \cite{abia_partitioned_1993}, or compose a number low order integrators to form a high order one \cite{yoshida_construction_1990}. The second reason is due to the nature of the optimal control Hamiltonian. In Hamiltonian mechanics, the move into the Lagrangian formulation is done by using Legendre transform, however, this transform is not bijective when the Hamiltonian is not hyperregular. One approach to solving Hamiltonian systems numerically is to move into the Lagrangian side and apply variational integrators \cite{marsden2001discrete}. Moreover, the application of structure--preserving integrators for optimal control was first done in the Lagrangian picture \cite{junge_discrete_2005} and further studied in \cite{lee_computational_2008,lee_optimal_2008,bloch_geometric_2009, kobilarov_discrete_2011}. The lack of hyperregularity impedes the move to the Lagrangian side and for this reason a Hamiltonian structure--preserving integrator is needed \cite{leok_discrete_2011}. The third reason is of practicality as our approach can be applied to both deterministic and stochastic optimal control problems, albeit with additional steps.  

\begin{figure}
    \centering
         \begin{subfigure}[b]{0.32\textwidth}
         \centering
         \includegraphics[width=\textwidth]{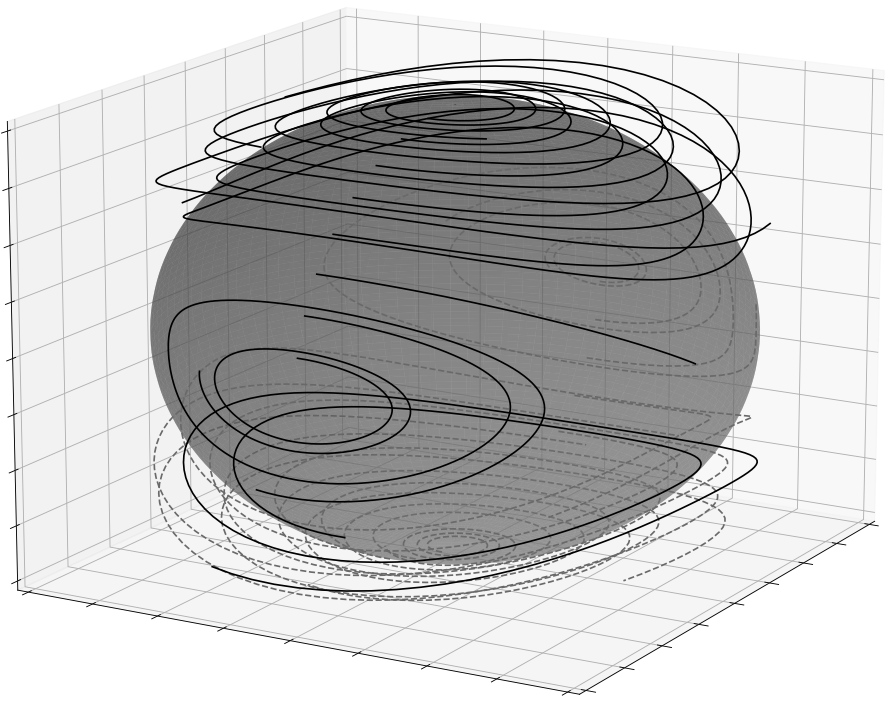}
         \caption{ }
         \label{fig:rigid_euler}
     \end{subfigure}
     \begin{subfigure}[b]{0.32\textwidth}
         \centering
         \includegraphics[width=\textwidth]{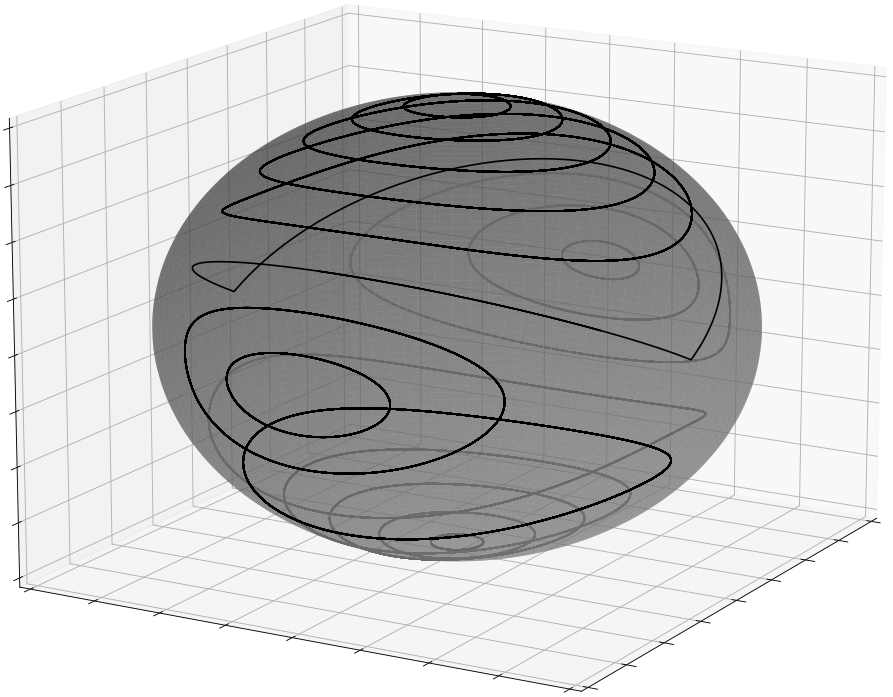}
         \caption{ }
         \label{fig:rigid_runge}
     \end{subfigure}
    \begin{subfigure}[b]{0.32\textwidth}
         \centering
         \includegraphics[width=\textwidth]{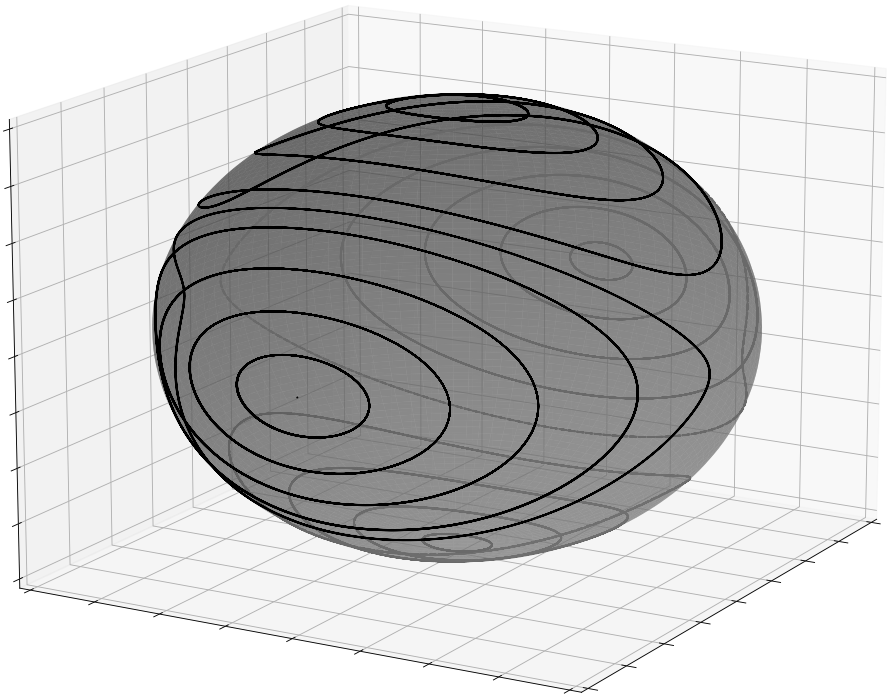}
         \caption{ }
         \label{fig:rigid_lphj}
     \end{subfigure}
        \caption{The solution of $SO(3)$ Rigid body, which is a Lie--Poisson system, computed using three different numerical methods to highlight the importance of structure--preserving integrators. The equation of motion is $\dot{\Pi} = \Pi \times \Omega$, with $\Pi = \mathbb{I}\Omega$ and $\mathbb{I} = \mathrm{diag}(1,2,3)$ is the inertia tensor. In figure \ref{fig:rigid_euler}, the Euler scheme is used to approximate the solution and it is evident that it does not stay on the sphere. The second scheme used is the $4^{th}$ order Runge--Kutta method, shown in figure \ref{fig:rigid_runge}, although it fares better than Euler method, it too does not preserve the structure. The last scheme is structure--preserving scheme that uses Lie--Poisson Hamilton--Jacobi theory of \cite{zhong1988lie}, the method is discussed in section \ref{sec:lphj_theory}, and the results that it yields, shown in figure \ref{fig:rigid_lphj}, are satisfactory. }
        \label{fig:three_spheres}
\end{figure}

\section{Preliminaries}
\subsection{Hamilton--Jacobi Theory}\label{sec:hamilton_jacobi_intro}
As it plays central role for the derivation of structure--preserving schemes, here, we give a brief summary of the Hamilton--Jacobi theory and the important concepts used throughout this article.  One source that we highly recommend is \cite{marsden2013introduction}. We consider the the space of all possible states of the dynamical system and it is called \emph{configuration manifold} $Q$ with coordinates $q = (q^1, \dots, q^n)$, where $n$ is the dimension of $Q$. The tangent space of the configuration bundle is $TQ$ and its dual is the cotangent bundle $T^{\ast}Q$, whose generalised coordinates are $(q,p) = (q^1, \dots, q^n, p^1, \dots, p^n)$.  In mechanics, there are a number of ways to describe the motion of the system and one of these is Newtonian mechanics. Despite the fact that the intuition behind Newtonian mechanics is easier to grasp, when there are physical and geometric constraints, it becomes difficult to model. Lagrangian mechanics was introduced to tackle this and it is a formulation that is widely used for constraint particle motion. Such formulation of mechanics requires $(q, \dot{q}) \in TQ$ to describe the motion of the system. An alternative formulation, and the one used here, is the Hamiltonian formulation which uses position and conjugate momenta  $(q, p) \in T^{\ast}Q$ to describe evolution of the system. Another important object in Hamiltonian mechanics is the Hamiltonian function $H: T^{\ast}Q \to \mathbb{R}$, which carries all the information of the dynamics of the system. The Hamiltonian, $H$, is related to the energy of the system and, in some cases, it coincides with the total energy. Once we have the Hamiltonian, $H$, and with fixed boundary terms $q(0)$ and $q(T)$, the variational principle states that 
\begin{align}
    \delta \int_0^T \left\langle p, \dot{q} \right\rangle - H(q,p) \, dt = 0,
\end{align}
and it yields Hamilton's equation
\begin{align}\label{eq:hamilton_eq}
    \frac{d}{dt} q = \frac{\partial H}{\partial p}(q,p), \quad \text{and} \quad \frac{d}{dt} p = -\frac{\partial H}{\partial q}(q,p).
\end{align}

Often case, Hamilton's equation are described as what is known as \emph{a Hamiltonian vector field} $X_H = (\frac{\partial H}{\partial p}, -\frac{\partial H}{\partial q})$ whose flow is a symplectomorhism, $\Phi_{X_H}: (q_0, p_0) \to (q, p)$ , and it maps the initial points $(q_0, p_0) \in T^{\ast}Q$ to the final points $(q(t), p(t)) \in T^{\ast}Q$ at time $t$. This mapping leaves the symplectic structure, $\omega = dq \wedge dp$ invariant, i.e.
\begin{align}\label{eq:symplectic_form}
\left(\Phi_{X_H} \right)^{\ast} \omega - \omega  = 0,
\end{align}
and this allows for constructing coordinate transformation that considerably simplifies the equation of motion. The mapping $\left(\Phi_{X_H} \right)^{\ast}$ is the pull-back of of $\omega$ by flow of $X_H$ and \eqref{eq:symplectic_form} can be easily proven by taking the time derivative and using Cartan's formula for Lie derivatives \cite{marsden2013introduction}.
\\

In fact, one choice of coordinates $(w, J)$, known as action-angle coordinates, chooses the momentum $J$ to be a constant of motion and this coordinate system is central to the celebrated \emph{Kolmogorov-Arnold-Moser theorem}.  Meanwhile, the symplectic form can be written as an exact form
\begin{align} \label{eq:omega_dtheta}
 \omega = \mathbf{d} \theta = dq \wedge dp, \quad \text{where} \quad \theta  = p \mathbf{d}q
\end{align}
where $\mathbf{d}$ is a map from $k$-forms on a manifold to $(k+1)$-forms on the same manifold known as \emph{exterior derivative} and it is linear and satisfies the product rule \cite{abraham_manifolds_1988}. In coordinates, if a $k$-form is given by
\begin{align*}
    \alpha = \alpha_{i_1, \dots, i_k} dx^{i_1} \wedge \dots \wedge dx^{i_k},
\end{align*}
the exterior derivative is given by
\begin{align*}
   \mathbf{d} \alpha = \sum_{j=1}^k\frac{\partial \alpha_{i_1, \dots, i_k}}{\partial x^j } dx^j \wedge dx^{i_1} \wedge \dots \wedge dx^{i_k}
\end{align*}
It has the properties that $\mathbf{d}^2 = 0$  and when $f$ is a scalar function, $\mathbf{d}f$ is equal to the differential of $f$. The relation \eqref{eq:omega_dtheta} holds is due to the fact that $\omega$ is a closed form, i.e. $\mathbf{d}\omega = 0$. This is important, as it allows us to write 
\begin{align}
\left(\Phi_{X_H} \right)^{\ast} \mathbf{d} \theta - \mathbf{d} \theta  = \mathbf{d} \left(  \left(\Phi_{X_H} \right)^{\ast}\theta  -  \theta  \right) = 0 = \mathbf{d}^2 S
\end{align}
and it implies there the difference between the one-form and its pullback by the flow of the Hamiltonian system is a closed one-form. This closed one-form is defined as an exterior derivative of the function, $S: Q \times Q \times \mathbb{R} \to \mathbb{R}$ , and it is known as \emph{generating function}. The generation of symplectic transformation from $S$ comes from the relation
\begin{align}
\mathbf{d} S = p \mathbf{d}q - p_0 \mathbf{d}q_0,
\end{align}
which yields
\begin{align}\label{eq:symplecto}
p(t) = \frac{\partial S}{\partial q}(q_0, q(t), t), \quad \text{and} \quad p_0 = \frac{\partial S}{\partial q_0}(q_0, q(t), t).
\end{align}
Luckily, the generating function satisfies a partial differential equation
\begin{align}\label{eq:Hamilton_Jacobi_eqn}
\frac{\partial S}{\partial t} + H \left(q, \frac{\partial S}{\partial q} \right) = 0.
\end{align}
known as \emph{Hamilton--Jacobi equation}. Solving Hamilton's equations and solving Hamilton--Jacobi equation is one and the same thing. The main difference between the two approaches is that the latter deals with functions defined on the configuration manifold, $Q$, while the former the dynamics belong to the cotangent bundle $T^{\ast}Q$.
\\

For certain systems, it is taxing to work with generating function $S$ that generates a canonical transformation from the old to the new coordinates, $(q_0, p_0)$ and $(q(t), p(t))$ respectively, and for this it is not unwise to consider a combination of old and new coordinates. The said combinations can be constructed using Legendre transformations and one transformation that is widely used is the type--II generating function
\begin{equation}
\begin{aligned}
S_2(q_0, p(t), t) &= S(q_0, q(t), t) + \left\langle q(t), p(t) \right\rangle.
\end{aligned}
\end{equation}
Substituting $S$ in terms of $S_2$ we obtain a Hamilton--Jacobi equation \eqref{eq:Hamilton_Jacobi_eqn} for each type of generating function. 

\begin{remark}
For numerical schemes based on generating functions, it is practical to use type--II generating function as it allows for initial condition to be $S_2(p_0, p(0), 0) = 0$ and, in the same time, satisfy the identity transformation required by the Hamilton--Jacobi theory.
\end{remark}
\subsection{Optimal Control Theory}
In this section, we give a brief overview of optimal control theory and formulating deep learning as an optimal control problem. The reader is invited to read of the following references \cite{agrachev2004control,jurdjevic_1996, FB/ADL:04supp} for more detailed exposition on the subject. Let $Q$ be a configuration manifold and a system's state, on the time interval $[0,T]$, is described by $q(t) \in Q$. The goal of optimal control is to move the system from the state $q(0) = q_0$ to the state $q(T) = q_T$, where $q_0, q_T \in Q$ are predetermined states. This is actualised by choosing $\theta = (u, b) \in \mathcal{U}$, where $\mathcal{U}$ is called \emph{the space of admissible controls}, such that the \emph{cost functional} 
\begin{align}\label{eq:cost_functional}
S(q, \theta) = \int_0^T \ell(q(t), \dot{q(t)}, \theta(t)) \, dt
\end{align}
is maximised/minimised and at the same time $q(t) \in Q$ satisfies 
\begin{align}
\frac{d}{dt}q = \sigma\left( u \cdot q + b\right),
\end{align}
where $q \in Q$, $\frac{d}{dt}q \in TQ$ and $\theta=(u,b) \in \mathcal{U}$ and $\sigma:Q \times \mathcal{U} \to \mathbb{R}$. The term $\ell:TQ \times \mathcal{U} \to \mathbb{R}$, in \eqref{eq:cost_functional}, is the \emph{loss function} or \emph{cost Lagrangian}. The choice of loss function depends on the application, and it is usually chosen as a measure of how far is the state of the system, $q(t)$, from a predetermined target. There are two methods that are widely used to compute $\theta$: dynamic programming and Pontryagin's Maximum Principle. The former breaks the problem into a smaller problem on the interval $[t, t + \Delta t]$ and uses the solution of the smaller problem to construct a solution on $[0,T]$ and it yields a solution which satisfies the necessary and sufficient condition for optimality. However, it is not considered in this article and instead we use Pontryagin's Maximum Principle. 
\\

The method in essence is finding a curve on the Pontryagin's bundle $T^{\ast}Q\times \mathcal{U}$, denoted by $(q(t), p(t), \theta(t))$, that minimises \eqref{eq:cost_functional} and it also satisfies the following Hamilton's equation:
\begin{align} \label{eq:Ham_equn_sec2}
\frac{d}{dt}q(t) = \frac{\partial H}{\partial p}, \quad \frac{d}{dt}p(t) = -\frac{\partial H}{\partial q}, \quad \mbox{and} \quad \frac{\partial H}{\partial \theta} = 0,
\end{align}
with the Hamiltonian function $H: T^{\ast}Q\times \mathcal{U} \to \mathbb{R}$, defined by $H(q,p,\theta) = \left\langle p,\sigma \left(q(t), \theta(t)\right) \right\rangle - \ell \left(q(t), \theta(t)\right)$.  These equations, \eqref{eq:Ham_equn_sec2}, are the stationary variations of Hamilton's principle
\begin{align} \label{eq:action_ham_principle}
0 =\delta \mathbb{S} = \delta \int_0^T \left\langle p, \dot{q} \right\rangle - H(q,p,\theta) \, dt.
\end{align}
The equation $\dot{p} = - \frac{\partial H}{\partial q}$ is the costate equation and also known as the adjoint equation and it has gained attention in deep learning research as it is the core of methods such as neural ordinary differential equation \cite{NIPS2018_7892, sun_neupde_2019, dupont2019augmented}. Solving such Hamiltonian system is an arduous task, thus numerical methods are required to approximate the solution. The application of numerical integrators for optimal control problems has been explored in \cite{junge_discrete_2005,lee_computational_2008,lee_optimal_2008,bloch_geometric_2009, kobilarov_discrete_2011}and the key idea is to use a discrete version of the variational principle. On the contrary, in this article, the idea is, instead of applying discrete variational principle, we use the structure--preserving numerical scheme derived using Hamilton--Jacobi theory, given by \eqref{eq:scheme_gen_func}, to approximate the solution of Pontryagin's Maximum Principle \eqref{eq:Ham_equn_sec2}. Concerning the initial conditions for $(q(t), p(t))$, we are only given $q_0$ and $q_T$ and for this reason an additional step is needed to find the correct $p^0$ that would make $(q^k, p^k, \theta^k)$, where $k = 0, \dots N_t$ satisfy $q^0 = q_0$ and $q^{N_t} = q_T$. There are a number of choices for the extra step, such as the shooting method where $p^0$ is chosen at random and then adjusted. The other approach is to use nonlinear programming, and the idea is that it makes the solution of the algebraic expressions \eqref{eq:scheme_gen_func} and the boundary conditions as the solution of a nonlinear system of equations.

\subsection{Deep Neural Network Training As Optimal Control Problem}\label{sec:opt_control_deep}
Now considering deep learning, we first need to a representation of the neural networks as a dynamical system. Artificial neural networks can be viewed as a group of simple sub-systems referred to as synapses connected in a specific way. Mathematically, a synapsis is modelled as the following transformation $Q_k \to Q_{k+1}$, without the loss of generality, in this article we solely consider the case where this transformation is an endomorphism, i.e. a transformation $Q \to Q$, given by
\begin{equation}\label{eq:discrete_neural_network}
\begin{aligned}
q_i^{k+1} &= F^n\left(q_i^{k},\theta^k \right), \quad \mbox{for} \quad k = 0, \ldots, N_t -1
\end{aligned}
\end{equation}  
where $q^{k+1}_i$ is the $i^{th}$ synapsis of the $k+1$ layer, $q^k \in Q$, $\theta^k = (u^k, b^k) \in \mathcal{U}$ is the network's parameters and the map $F^k: Q \times \mathcal{U} \to Q$ is the state transition map.  The vector $q^{k+1} = (q^{k+1}_1, \ldots, q^{k+1}_{N_l}) \in Q \subset \mathbb{R}^{N_l}$ constitute a single layer in the network. For general cases, the configuration space $Q$ need not be vector spaces, they can be Hilbert spaces  or differentiable manifold, and in that case \eqref{eq:discrete_neural_network} can be regarded as a pullback operation from $Q_k$ to $Q_{k+1}$. The choice of the state transition map determines the type of network and here we consider a transition of the type 
\begin{align}\label{eq:discrete_resNet}
q_i^{k+1} &= F \left(q^k_i,  \theta^k\right) = q^k_i + \alpha \sigma \left( \sum\limits_{j=0}^{N_l}u^k_{i,j}q^k_{j} + b^k_i \right), \quad \alpha \in \mathbb{R}_{>0},
\end{align}
and this renders \eqref{eq:discrete_neural_network} as residual neural network, ResNet henceforth. The function $\sigma:  Q \times \mathcal{U} \to Q$ is the activation function and it is applied element-wise  and few of the most commonly used ones are: linear rectifier, denoted by ReLU, hyperbolic tan, sigmoid function, and softmax. Here,  $u^k_{i,j}$ is the element of the $i^{th}$ row and $j^{th}$ column of weight matrix and $b^k_i$ is the $i^{th}$ element of bias vector for the $k^{th}$ layer.  We do not specify $\theta^k = (u^k, b^k)$ for $k=0, \ldots, N_t$, before hand, at first, random weights and biases are sampled form a distribution function. Finding the values for $\theta^k$ to make the neural network model a specific system is the focus of training methods studied in the field of machine learning. 
\\

Equation \eqref{eq:discrete_resNet} is viewed as Euler discretisation of a continuous dynamical system \cite{e_proposal_2017, e_mean-field_2018, NIPS2018_7892, ruthotto_deep_2019}. The move to the continuous domain can be interpreted as taking the limit when the number of hidden layers approaches infinity, $N_l \to \infty$. The result is a dynamical system defined on the manifold $Q$ of the form 
\begin{align}\label{eq:contSystem}
    \frac{d}{dt} q(t) = \sigma(q(t), \theta(t)), \quad t \in [0, T], \quad \mbox{and} \quad q(0) =q_0,
\end{align}
where $q_0 \in Q$. 
\begin{remark}
In some literature \cite{NIPS2018_7892, dupont2019augmented}, the limit of ResNet, \eqref{eq:contSystem}, is referred to as \emph{neural ordinary differential equation}. In this article, we substitute this naming with continuous ResNet. 
\end{remark}

Training the neural network can be treated as an optimal control problem with \eqref{eq:contSystem} as a constraint, and the dataset $ \{(q_0^{(i)}, c^{(i)} ) \}_{i = 0, \dots, N}$ specifies the initial conditions and the final conditions  \cite{li2017maximum}.   We define $\mathcal{C}$ as the space of target data.  Thus, we state the training of continuous ResNet as follows
\begin{problem}\label{prob:opt_training_det}
Given the training set  $ \{(q_0^{(i)}, c^{(i)} ) \}_{i = 0, \dots, N}$, compute $\theta = (u, b)$ such that it minimises 
\begin{align}\label{eq:control_functional_det}
S = \int_0^T \ell_d \left(q(t), \theta(t)\right) \, dt + \sum_{i=0}^N L( \pi((q(T)^{(i)}), c^{(i)} ), 
\end{align}
such that it satisfies
\begin{align}
\frac{d}{dt} q(t) = \sigma\left( u(q(t)) + b(t) \right),
\end{align}
and $q(0) = q_0^{(i)} \, \text{and} \, \pi(q(T)) =  c^{(i)}$ for $i = 0, \dots, N$.
\end{problem}
Here, where $\ell : Q \times \mathcal{U} \to \mathbb{R}$ is the cost Lagrangian and in machine learning it acts as a regularisation term.  The last term of the functional $S$ is the terminal condition and it is a metric for measuring  how far the output of \eqref{eq:control_functional}, $q(T)^{(i)}$, whose starting point is $q(0) = q_0^{(i)}$,  is from the target data $c^{(i)}$. A suitable candidate is  $\sum\limits_{i=0}^N L( \pi(q(T)^{(i)}), c^{(i)} )$, where $L: \mathcal{C} \times \mathcal{C} \to \mathbb{R}$ and the projection $\pi: Q \to \mathcal{C}$ and in this article it is chosen as $\pi(x) = \frac{1}{1 + e^{-x}}$. As for $L$, in many classification problems, a common and a suitable choice is the categorical cross entropy loss function.

\section{Stochastic Optimal Control for Deep Learning}
In this section, we consider a continuous residual neural network that takes into consideration the network's uncertainties. These uncertainties are accounted for by modelling them as a stochastic process and for the time-being, we only consider Wiener processes. That means the dynamical model for continuous ResNet, given by \eqref{eq:contSystem}, is now a stochastic differential equation. The addition of stochastic terms not only render the model robust for potential variations, but also it is very rich mathematically as shall we explore in this section. 
\\

Similar to the previous sections, we present an algorithm for training deep learning using a generating function for stochastic continuous ResNet. As the dynamical system is now stochastic, the application of the Pontryagin's Maximum Principle used in this article cannot be done directly. For this we require stochastic optimal control and especially mean--field type optimal control \cite{bensoussan_mean_2013}, and, in this paradigm, the solution to the optimal control problem satisfies a system of partial differential equations instead of ordinary differential equations. The idea of using stochastic optimal control for deep learning was studied in \cite{e_mean-field_2018, persio_deep_2021}, however, the main difference is we use elimination theorem of \cite{doi:10.1098/rspa.2007.1892} to reduce the system of partial differential equations to hydrodynamic equation with advected quantity. 
\\

One major consequence of mean--field optimal control is that we can no longer utilise the Hamilton--Jacobi theory used in section \ref{sec:hamilton_jacobi_intro}. However, there are two possible ways to construct structure--preserving schemes for optimal control: using multisymplectic integrators and using integrators based on generating functions on Poisson manifolds. In this article, we only focus on the latter, as the former is studied in \cite{G2021}. The key here is to use Poisson map analogous to the symplectic transformation and we, first, reduce the system of partial differential equations that we obtain by mean--field type optimal control to a hydrodynamics equation. Then we restate the said hydrodynamics equation as a Lie--Poisson equation, which is the same equation but from the Hamiltonian viewpoint. The Poisson geometry counterpart for the Hamilton--Jacobi theory was introduced in \cite{zhong1988lie} and it is known as the Lie--Poisson Hamilton--Jacobi theory. 
\\

We give a brief overview of the theory in section \ref{sec:lphj_theory} and we introduce a variant of the theory that concerns  semidirect products. In section \ref{sec:pseudospectral} we introduce the algebras of pseudodifferential symbols whose dual space is where the Lie--Poisson equation evolves. In section \ref{sec:Discrete_LP}, we use algebras of pseudodifferential symbols on discrete meshes to derive numerical scheme based on generating function on Poisson space is presented in proposition \ref{prop:numerical_LP}.

\subsection{Hamilton--Jacobi for Systems with Symmetry}\label{sec:lphj_theory}
Having demonstrated how Hamilton--Jacobi theory is used as a formulation for Hamiltonian systems on $Q$, here we consider a special case when the Hamiltonian is $G$-invariant. What we mean by that, we consider a configuration manifold $Q$ and when $g \in G$ is a Lie group action acting on the Hamiltonian $H:T^{\ast}Q \to \mathbb{R}$, the following holds
\begin{align*}
H(g \cdot q, (g)^{\ast}p) = H(q,p), \quad (q,p) \in T^{\ast}Q.
\end{align*}
This allows us to perform the Lie--Poisson reduction \cite{marsden2013introduction}. We will not cover reduction theory here, the reader is advised to refer to \cite{marsden2013introduction, marsden_semidirect_1984, marsden2007hamiltonian}. In essence, the invariance allows us to choose action such that $T^{\ast}Q/G$ and the resultant quotient space is isomorphic to dual of the Lie algebra $\mathfrak{g}^{\ast}$. For example, when $Q = SO(3)$ and $G =  SO(3)$, the quotient space $T^{\ast} SO(3)/ SO(3)$ is isomorphic to $\mathfrak{so}^{\ast}(3)$, the dual of the Lie-algebra associated to $ SO(3)$. In general, the dual of Lie algebra $\mathfrak{g}^{\ast}$ is an example of what is known as a Poisson manifold \cite{marsden_semidirect_1984}. A manifold $(P, \{ \cdot, \cdot \})$ is said to be a \emph{Poisson manifold} when a manifold $P$ is equipped with a bilinear form $\{ \cdot, \cdot \}$ that has the properties that is for any function on $C^{\infty}(P)$, the form is anti-symmetric, satisfies both the Jacobi identity and Leibniz's rule. The bilinear form $\{ \cdot, \cdot \}$ is known as \emph{the Poisson structure}. Let $(P_1, \{ \cdot, \cdot \}_1)$ and $(P_2,\{ \cdot, \cdot \}_2)$ be two Poisson manifolds, then a map $\phi: P_1 \to P_2$ that satisfies $\{ f \circ \phi, h \circ \phi \}_1 = \{ f, h \}_2\circ \phi$ for any $f, h \in C^{\infty}(P_2)$ is called a \emph{Poisson map}. A detailed exposition of reduction is beyond the scope of this article, however, it is only summarised. For the dynamics, whose Hamiltonian is $H:T^{\ast}Q \to \mathbb{R}$, its integral curves when projected to a Poisson manifold $P$ using the map $\pi_P:T^{\ast}Q \to P$, coincide with the integral curves of dynamics, whose Hamiltonian is $H_P:P \to \mathbb{R}$. We can say that the flow $F$ of the Hamiltonian $H$ and the flow $F_P$ of $H_P$ are related by $\pi_P \circ F = F_P \circ \pi_P$. 
\\
\[ \begin{tikzcd}
T^{\ast}Q \arrow{r}{F} \arrow[swap]{d}{ } & T^{\ast}Q \arrow{d}{ } \\%
P \arrow{r}{F_P}& P
\end{tikzcd}
\]
The Hamiltonian dynamics on $T^{\ast}Q$ induce Lie--Poisson equation on $\mathfrak{g}^{\ast}$
\begin{align}
\frac{d}{dt}F = \{F, H_P \},
\end{align}
which is equivalent to 
\begin{align}
\frac{d}{dt} \mu = \mathrm{ad}^{\ast}_{\frac{\delta H_P}{\delta \mu}}\mu,
\end{align}
where $ \mathrm{ad}^{\ast}: \mathfrak{g} \times \mathfrak{g}^{\ast} \to \mathfrak{g}^{\ast}$ is the coadjoint map. 
\\

The Lie--Poisson reduction can also be used for the Hamilton--Jacobi theory \cite{zhong1988lie}. Starting with the Hamilton--Jacobi theory
\begin{align}\label{eq:HJ_theory}
\frac{\partial S}{\partial t}(q, q_0) + H \left(q, \frac{\partial S}{\partial q} \right) = 0,  \quad p_0 = -\frac{\partial S}{\partial q_0}, \, p(t) = \frac{\partial S}{\partial q},
\end{align}
where $S: Q \times Q \to \mathbb{R}$ is the generating function, the Hamiltonian $H:T^{\ast}Q \to \mathbb{R}$ is a $G$-invariant Hamiltonian. The generating function yields a canonical transformation from $(q_0, p_0) \to (q(t), p(t))$. Let $S_L(g) = S(q^{-1}q, q^{-1}q_0) = S(e, g)$, where $g = q^{-1}q_0$, the Lie--Poisson reduction reduces \eqref{eq:HJ_theory} to 
\begin{align}\label{eq:LPHJ_theory}
    \frac{\partial S_L}{\partial t}(g) + H \left(-\mathbf{J}_L \circ \mathbf{d}S_L \right) = 0, \quad \mu_0 = TL_g^{\ast}\mathbf{d}_g S_L, \quad \mu = TR_g^{\ast} \mathbf{d}_g S_L,
\end{align}
with $\mu_0, \mu \in \mathfrak{g}^{\ast}$. Here we have $S_L :Q \to \mathbb{R}$, $g \in G$ is the Lie group action, $TL_g^{\ast}$ and $TR_g^{\ast}$ are the left and right cotangent lift of Lie group action $g$. The map $\mathbf{J}_L:T^{\ast}Q \to \mathfrak{g}^{\ast}$ is the \emph{the momentum map} for the cotangent lift and  it is equal to
\begin{align}
  \mathbf{J}_L(\alpha) =   T_eR_g^{\ast}\alpha,
\end{align}
while $\mathbf{J}_R:T^{\ast}Q \to \mathfrak{g}^{\ast}$ is given by
\begin{align}
  \mathbf{J}_R(\alpha) =   T_eL_g^{\ast}\alpha.
\end{align}
The partial differential equation and two expressions in \eqref{eq:LPHJ_theory} are known as \emph{Lie--Poisson Hamilton--Jacobi theory} introduced in \cite{zhong1988lie}. In fact, the integrator that was used in figure \ref{fig:three_spheres}, to showcase one of the advantages of structure--preserving integrators, uses the discrete version \eqref{eq:LPHJ_theory} on $ SO(3)$.
\\

Now, the generating function $S_L$, generates a Poisson map $\phi: \mu_0 \to \mu$, which is the analogue of symplectic map on the cotangent bundle $T^{\ast} Q$. This reduction is not only valid for semisimple Lie groups such as $SO(3)$, it can be applied to semidirect products. First, let us recall some basic facts about semidirect products. Let $S = G \oplus V$ be a semidirect product space with Lie group $G$ and vector space $V$. Let $\zeta : G \to \mathrm{Aut}(V)$ be the representation of the Lie group $G$ in $V$, and $\zeta': \mathfrak{g} \to \mathrm{End}(V)$ is the Lie algebra representation in $V$.  The multiplication of $S$ is defined as
\begin{align*}
(g_1, v_1)(g_2, v_2) = \left(g_1 g_2, v_1 + \zeta(g_1)v_2 \right),
\end{align*}
where $(g_1, v_1), (g_2, v_2) \in S$. The corresponding Lie algebra $\mathfrak{s} = \mathfrak{g} \times V$ has the Lie bracket
\begin{align*}
\left[ (\xi_1, v_1), (\xi_2, v_2)\right] = \left( \left[ \xi_1, \xi_2\right], \zeta'(\xi_1) v_2 - \zeta'(\xi_2) v_1 \right),
\end{align*}
where $(\xi_1, v_1), (\xi_2, v_2) \in \mathfrak{s}$. The left action of $S$ on to itself is defined as
\begin{align}
L_{(g, v)}(h, u) = \left( g h, v + \zeta(g)u \right),
\end{align}
and tangent lift of the left action is 
\begin{align}
T_{(h, u)}L_{(g, v)} (v_{h}, u, w) = \left( T_{h}L_{g}(v_{h}, v + \zeta(g)u, \zeta(g) w \right), 
\end{align}
where $(v_{h}, u, w) \in T_{(h, u)}\left(G \oplus V \right)$. The tangent lift at the element $(g,v)(h,u) \in S$ and we are introducing it, so we can define left action momentum map $\mathbf{J}_L$ later. 
\begin{align}
T_{(g, v)(h, u)}L_{(g, v)^{-1}} (v_{gh}, v + \zeta(g)u, w) = \left( T_{gh}L_{g^{-1}}(v_{gh}, u, \zeta(g^{-1}) w \right),
\end{align}
for any $(v_{gh}, v + \zeta(g)u, w) \in T_{(g, v)(h, u)}\left(G \oplus V \right)$. The cotangent lift is 
\begin{align}\label{eq:leftcotangent}
\left( T_{(g, v)(h, u)}L_{(g, v)^{-1}}  \right)^{\ast} \left( \alpha_{h}, v, a \right) = \left( \left( T_{gh}L_{g^{-1}}\right)^{\ast} \alpha_{h}, u + \zeta(g) v, \zeta_{\ast}(g) a \right),
\end{align}
where $(\alpha_{h}, v, a) \in T^{\ast}_{(h, u)} \left(G \oplus V \right)$. 
On the other hand, the right action of $S$ on itself is 
\begin{align}
R_{(g, v)}(h, u) =\left( h g, u + \zeta(h)v\right),
\end{align}
and the tangent lift is 
\begin{align}
T_{(h, u)}R_{(g, v)} (v_{h}, u, w) = \left( T_{h}L_{g}(v_{h}, u + \zeta(h)v, \zeta(h) w \right),
\end{align}
and the lift over the point $(g,v)(h,u) \in S$, 
\begin{equation}
\begin{aligned}
T_{(h,u)(g,v)}R_{(g, v)^{-1}}& (v_{hg}, u + \zeta(h)v, w)  =\\
&\left( T_{hg}R_{g^{-1}}(v_{hg}, v, w - T_{hg}\zeta(v_{hg}) \cdot \zeta(g^{-1})v \right),
\end{aligned}
\end{equation}
and the cotangent lift is
\begin{equation}\label{eq:rightcotangent}
\begin{aligned}
\left( T_{(h,u)(g,v)}R_{(g, v)^{-1}}\right)^{\ast} &\left( \alpha_{h}, v, a \right) = \\
&\left( \left( T_{hg}R_{g^{-1}}\right)^{\ast}  \alpha_{h} - \mathbf{d} f_{\zeta(g^{-1})v}(hg), v + \zeta(h) v, a \right)
\end{aligned}
\end{equation}
With that we can write the left momentum map, $\mathbf{J}_L: T^{\ast}\left(G \oplus V \right) \to \mathfrak{s}^{\ast}_{+}$ and the right momentum map, $\mathbf{J}_R: T^{\ast}\left(G \oplus V \right) \to \mathfrak{s}^{\ast}_{-}$ by using \eqref{eq:leftcotangent} and \eqref{eq:rightcotangent} with $(g,v) = (g,v)^{-1}$ and $(h,u) = (g, v)$ and that gives
\begin{align}
    \mathbf{J}_L &= \left( T_{(e, 0)}R_{(g, v)}  \right)^{\ast} \left( \alpha_{g}, u, a \right) = \left( \left( T_eR_g\right)^{\ast} \alpha_g + (\zeta'_u)^{\ast} a, a\right), \\
    \mathbf{J}_R &= \left( T_{(e, 0)}L_{(g, v)}  \right)^{\ast} \left( \alpha_{g}, u, a \right) = \left( \left( T_eL_g\right)^{\ast} \alpha_g , \zeta^{\ast}(g) a\right).
\end{align}
\begin{theorem}\label{prop:SemiLPHJ}
Given the left $G$-invariant Hamiltonian, $H$, then the left reduced Lie--Poisson Hamilton--Jacobi equation for semidirect products for a function $S_L : G \oplus V \times \mathbb{R} \to \mathbb{R}$ is the following:
\begin{equation}
\begin{aligned}\label{eq:semiLPHJeqn}
\partial_t S_L((g,v),t) &+ H\left( TR^{\ast}_{g}\mathbf{d}_{g}S_L + (\zeta'_v)^{\ast} \mathbf{d}_v S_L , \mathbf{d}_{v}S_L \right)  = 0 \\
\end{aligned}
\end{equation}
The solution of equation \eqref{eq:semiLPHJeqn}, known as \textbf{Lie--Poisson Hamilton--Jacobi equation}, generates a Poisson map that defines the flow of the reduced Hamiltonian $H_{r}$. The Poisson map transforms the initial momentum map $\mu_0 \to \mu(t)$, where
\begin{align} 
(\mu_0, \Gamma_0) =&  \left(-TL_g^{\ast} \mathbf{d}_g S_L, \zeta^{\ast}(g) \mathbf{d}_v S_L  \right), \label{eq:LPHJ_left_action} \\
\text{and} \quad (\mu(t), \Gamma(t)) =& \left( TR^{\ast}_{g}\mathbf{d}_{g}S_L + (\zeta'_v)^{\ast} \mathbf{d}_v S_L , \mathbf{d}_{v}S_L \right),\label{eq:LPHJ_right_action}
\end{align}
where $g \in G$ and $v \in V$.
\end{theorem}
For construction of a numerical scheme, the steps followed are similar to the ones taken in section \ref{sec:hj_for_dl}. First, the equation is approximated numerically in time, using a numerical scheme, such as Euler method or Runge--Kutta method, while maintaining that is it is parameterised in $(g,v)$. For example, a first order in $t$ approximation is given by
\begin{align}
    S_L((g,v)) \approx S_0((g,v)) + \Delta t H\left( TR^{\ast}_{g}\mathbf{d}_{g}S_0 + (\zeta'_v)^{\ast} \mathbf{d}_v S_0 , \mathbf{d}_{v}S_0 \right),
\end{align}
where $S_0$ is the initial condition of $S_L$ and it is chosen such that it satisfies the identity mapping. The next step is to then use this approximation of $S_L$ and substitute it into equation \eqref{eq:LPHJ_left_action} and it allows to obtain the group action $(g,v)$ that would result in $(\mu_0, \Gamma_0)$. Once the action $(g,v)$ is found,  it is then used in \eqref{eq:LPHJ_right_action} to compute $(\mu(t), \Gamma(t))$.

\subsubsection{Lie Algebra of Pseudodifferential Symbols}\label{sec:pseudospectral}
Here we give briefly go over the basics of Lie group whose Lie algebra are pseudodifferential symbols. A profound study of such groups is not required at the moment as we only require the basics. For detailed exposition on this topic, the reader can refer to \cite{blaszak2012multi,khesin_poisson-lie_1995, khesin2008geometry}. This Lie algebra is required as it allows for the application of Lie--Poisson Hamilton--Jacobi to Lie--Poisson partial differential equations. 
\\

We begin by defining what we mean by an operator of order $m$,
\begin{align}\label{eq:pso_operator}
P = \sum\limits_{|\alpha| \leq m} a_{\alpha}(x) D^{\alpha},
\end{align}
where $D$ is a differential operator and $\alpha$ is a multiindex and $a_{\alpha} \in H^s(\mathbb{R}^n)$. The \emph{symbol} of the operator $P$ is a polynomial in $\xi$ variable
\begin{align}
p(x, \xi) = \sum\limits_{|\alpha| \leq m} a_{\alpha}(x) \xi^{\alpha},
\end{align}
and it is used a way to describe \eqref{eq:pso_operator} as a polynomial of Fourier multipliers. 
In this section we denote the algebra of pseudodifferential  symbols by $\Psi \mathrm{DS}_{m}$, where the $m$ subscript indicates the order. The union of such algebras $\Psi \mathrm{DS}_{m}$ forms an infinite dimensional algebra $\Psi \mathrm{DS} = \bigcup\limits_{k}\Psi \mathrm{DS}_{k}$ and it is closed under multiplication defined by
\begin{align}\label{eq:pso_mul}
U(x, \xi) \circ M(x, \xi) = \sum\limits_{n \geq 0} \frac{1}{n!} \left( \frac{d^n}{d \xi^n} U(x, \xi)  \right)\left(  \frac{d^n}{d x^n} M(x, \xi) \right), 
\end{align}
where $U \in \Psi \mathrm{DS}_{m}$ and $M \in \Psi \mathrm{DS}_{n}$. The multiplication, $U(x, \xi) \circ M(x, \xi)$, belongs to the group $\Psi \mathrm{DS}_{n+m} \subset \Psi \mathrm{DS}$, i.e. multiplication is a closed group operation. The Lie bracket of the Lie algebra $\Psi \mathrm{DS}$ is given by
\begin{equation}
\begin{aligned}
\left[U, M \right] &= U \circ M - M \circ U \\
&= \sum\limits_{n \geq 0} \frac{1}{n!} \left\{ \left( \frac{d^n}{d \xi^n} U(x, \xi)  \right)\left(  \frac{d^n}{d x^n} M(x, \xi) \right) - \left( \frac{d^n}{d \xi^n} M(x, \xi)  \right)\left(  \frac{d^n}{d x^n} U(x, \xi) \right)\right\}.
\end{aligned}
\end{equation}
The algebra $\Psi \mathrm{DS}$ can be written as the direct sum of two other subalgebras, i.e. 
\begin{align}\label{eq:pso_decompose}
\Psi \mathrm{DS} = \Psi \mathrm{DS}_{\mathrm{DO}} \oplus \Psi \mathrm{DS}_{\mathrm{INT}}, 
\end{align}
where
\begin{align*}
\Psi \mathrm{DS}_{\mathrm{DO}_{m}} := \left\{ \sum\limits_{\alpha= 0}^m a_{\alpha}(x) \partial^{\alpha} \right\}, \quad \text{and}\ \quad \Psi \mathrm{DS}_{\mathrm{INT}_{m}} := \left\{ \sum\limits_{\alpha= -m}^{-1} a_{\alpha}(x) \partial^{\alpha} \right\},
\end{align*}
and $\Psi \mathrm{DS}_{\mathrm{DO}} = \bigcup\limits_{m \geq 0} \Psi \mathrm{DS}_{\mathrm{DO}_{m}}$ and $\Psi \mathrm{DS}_{\mathrm{INT}} = \bigcup\limits_{m \geq 0} \Psi \mathrm{DS}_{\mathrm{INT}_{m}}$. 

Usually for Lie algebras, there's an associated Lie group and to map the Lie algebra to the Lie group an exponential map is used.  The map needs to be bijective and here is defined as the solution of the following differential equation
\begin{align}\label{eq:exponential_map_def}
\frac{d}{ds} G(x,s) = U \circ G(s), 
\end{align}
where $G$ is the group element and $U$ belongs to the Lie algebra, and the initial condition is chosen as $G(0) = \mathrm{I_d}$. The differential equation \eqref{eq:exponential_map_def} actually yields a system of differential equations, in most cases it is a triangular system, for each coefficients of symbols of $G$ and solving this system is required to obtain the coefficients of the group element $G$ \cite{khesin_poisson-lie_1995}. Here we use the $\Psi G$ as the group whose vector fields at the identity are the $\Psi \mathrm{DS}$ symbols. Basically, $\Psi G$ is a group whose elements are solutions of \eqref{eq:exponential_map_def}. The reasoning behind not considering it a full Lie group is that the full algebra of symbols $\Psi \mathrm{DS}$ does not have a corresponding Lie group nor its subalgebra $\Psi \mathrm{DS}_{\mathrm{DO}}$. However, for the subalgebra $\Psi \mathrm{DS}_{\mathrm{INT}}$ does have an associated Lie group $G_{INT}$ and the exponential map is bijective $\mathrm{exp}: \mathfrak{g}_{\mathrm{INT}}  \to G_{\mathrm{INT}}$. We denote the algebra $\Psi \mathrm{DS}_{\mathrm{INT}}$ by $\mathfrak{g}_{\mathrm{INT}}$ henceforth. The pairing between the Lie algebra $\mathfrak{g}_{\mathrm{INT}} $ and its dual $\mathfrak{g}_{\mathrm{INT}}^{\ast}$, $\left\langle \cdot, \cdot \right\rangle: \mathfrak{g}_{\mathrm{INT}} \times \mathfrak{g}_{\mathrm{INT}}^{\ast} \to \mathbb{R} $ is defined by
\begin{align}
\left\langle M, U \right\rangle = \mathrm{Tr} \left( M \circ U \right) = \int \mathrm{res} \left( M \circ U  \right), \quad U \in \mathfrak{g}_{\mathrm{INT}}, \, M \in \mathfrak{g}_{\mathrm{INT}}^{\ast}.
\end{align} 
and we actually find that $\mathfrak{g}_{\mathrm{INT}}^{\ast} \simeq \mathfrak{g}_{\mathrm{DO}}$, where $\mathfrak{g}_{\mathrm{DO}}$ denotes the Lie algebra of differential operators $\Psi \mathrm{DS}_{\mathrm{DO}}$. For Hamiltonian systems, we need to compute functional derivatives. Let $M$ be an element of  $\mathfrak{g}^{\ast}_{\mathrm{INT}}$, and consider the functional $F: \mathfrak{g}^{\ast}_{\mathrm{INT}} \to \mathbb{R}$, its functional derivative is obtained using the following pairing
\begin{align}
DF(M) \cdot M = \left\langle \frac{\delta F}{\delta M}, M \right\rangle = \sum\limits_i \int \frac{\delta F}{\delta m_i}m_i \, dx,
\end{align}
thus we obtain,
\begin{align}\label{eq:func_deriv}
\frac{\delta F}{\delta M} = \sum\limits_{i = 1}^m  \frac{\delta F}{\delta m_i} \xi^{-i}.
\end{align}
Using the functional derivatives, one can define the Lie--Poisson bracket by
\begin{align}\label{eq:lie_poisson_pso}
\left\{ F, H \right\}\left( M \right) = \left\langle M, \left[ \frac{\delta F}{\delta M}, \frac{\delta H}{\delta M} \right] \right\rangle = \int \mathrm{res} \left( M \circ \left[ \frac{\delta F}{\delta M}, \frac{\delta H}{\delta M} \right]   \right),
\end{align}
where $H$ and $F$ are functionals that depend on the coefficients of the pseudodifferential symbols.  
\\

To apply theorem \ref{prop:SemiLPHJ}, for deep learning, we first need to compute the cotangent lifts of the action on the dual of the Lie algebra. Let $A$ be an element of the group $G_{INT}$. The multiplication operator is defined by \eqref{eq:pso_mul} and to compute the tangent lift, let $B \in \Psi G_{INT}$ such that $\left.\frac{d}{ds} \right|_{s = 0} B(s) = X$, where $X \in \mathfrak{g}_{\mathrm{INT}}$. Multiplying $A$ with $B$ and taking the derivative with respect to $s$ and evaluate at $s = 0$, we have
\begin{align}\label{eq:left_tangent}
	TL_AX = \left.\frac{d}{ds} \right|_{s = 0} A(t) \circ B(s) =  A(t) \circ X, 
\end{align}
and
\begin{align}\label{eq:right_tangent}
	TR_AX =  \left.\frac{d}{ds} \right|_{s = 0}  B(s) \circ A(t) = X \circ A(t).
\end{align}
To compute the cotangent lifts, $TL^{\ast}_A$ and $TR^{\ast}_A$, are computed by pairing the tangent lifts to an element $\alpha \in \mathfrak{g}_{\mathrm{INT}}^{\ast}$, the pairing is done using the trace operator, i.e
\begin{align}\label{eq:left_cotangent}
\mathrm{Tr} \left(TL^{\ast}_A \alpha \circ X\right)  &= \mathrm{Tr} \left( \alpha \circ TL_AX\right) =  \int \mathrm{res} \left( \alpha \circ A(t) \circ X\right),
\end{align}
and
\begin{align}\label{eq:right_cotangent}
\mathrm{Tr} \left(TR^{\ast}_A \alpha \circ X\right)  &= \mathrm{Tr} \left( \alpha \circ TR_AX\right) =  \int \mathrm{res} \left( \alpha \circ X \circ A(t) \right).
\end{align}
These yield the following cotangent lift of the left and right action
\begin{align}
TL^{\ast}_{A } \alpha &= \alpha \circ A(t), \label{eq:left_cotangent_lift}\\
TR^{\ast}_{A } \alpha &= A(t) \circ \alpha \label{eq:right_cotangent_lift}.
\end{align}

\subsection{Hydrodynamics Formulation of Deep Learning}\label{sec:deep_poisson}
From the basic iterative relation of deep neural networks, we can interpret them, in the language of geometric mechanics, as applying a pullback to a curve on a manifold. It is very reminiscent of the particle relabelling symmetry of fluids labels, i.e the action of the group of diffeomorphism, denoted by $\mathrm{Diff}(Q)$, on the Lagrangian labels \cite{arnold1966geometrie, 10.2307/1970699}. The vector field associated with $\mathrm{Diff}(Q)$ is denoted by $\mathfrak{X}(Q)$. The neurons of the hidden layers can be viewed as Lagrangian labels 
\begin{align}
\frac{d}{dt}q(x,t) = \sigma(u \circ q(x,t)), 
\end{align}
where $q \in \mathrm{Diff}(Q)$ and $u \in \mathfrak{X}(Q)$, and each time they propagate through to the next layer, a transformation is applied to the labels. Let $(\Omega, \mathcal{F}, \mathbb{P})$ be the probability space, with $\Omega$ the sample space, and $\mathcal{F}$ is $\sigma$-algebra of measurable events of $\Omega$ and $\mathbb{P}$ is the probability measure on $(\Omega, \mathcal{F})$.  We consider the stochastic differential equation for continuous ResNet  
\begin{align}\label{eq:sde_resnet_0}
d_tq(t) =  (u \circ q) \, dt + \nu \sum\limits_k \xi_k dW_t, \quad q(0) = q_0, \quad \text{and} \quad q(T) = q_T,
\end{align}
where $u \in \mathfrak{X}(Q)$ is a vector field that plays the role of network parameters, and $\xi_k \in \mathfrak{X}(Q)$ are vector fields that couple the Brownian motion $W_t^k$ to $\mathfrak{X}(Q)$ and $\nu \in \mathbb{R}$. Notice in \eqref{eq:sde_resnet_0} the  activation function was removed because it is redundant in this form. Here, the network parameters are specified to be a vector field on the group of diffeomorphisms on $Q$ and can be viewed as a combination of the weight matrix and activation function. Since the system is stochastic, the cost functional \eqref{eq:control_functional} cannot be used in its current form, instead, we use its expected value. Thus, we can state the training problem of stochastic continuous ResNet as follows: 
\begin{problem}\label{prob:opt_training}
Let $(\Omega, \mathcal{F}, \mathbb{P})$  and given the training set  $ \{(q_0^{(i)}, c^{(i)} ) \}_{i = 0, \dots, N}$, compute $u \in $ such that it minimises 
\begin{align}\label{eq:control_functional}
S = \int_{Q}\int_0^T \rho(x,t) \ell_d \left(u(t)\right) \, dt \, dx + \sum_{i=0}^N \int_{Q} \rho(x,t)  L( \pi((q(T)^{(i)}), c^{(i)} ) \, dx , 
\end{align}
where $\rho(x,t)$ is the probability density function. The solution should make $q(t)$ satisfy
\begin{align}\label{eq:sde_resnet}
d_tq = u\circ q \, dt + \nu\sum\limits_k \xi_k dW_t^k,
\end{align}
and $q(0) = q_0^{(i)} \, \text{and} \, \pi(q(T)) =  c^{(i)}$ for $i = 0, \dots, N$. 
\end{problem}
The solution method for this optimal control problem is mean--field optimal control, described in \cite{bensoussan_mean_2013}. Instead of controlling \eqref{eq:sde_resnet}, the equation for the evolution of probability density function associated with $q(t)$, known as the Fokker--Planck equation, 
\begin{align}
\partial_t \rho(x,t) + \nabla \cdot \left( u(x,t) \rho(x,t) \right) - \frac{\nu^2}{2}\nabla \cdot (\nabla \rho(x,t)) = 0, \quad \rho(x,0) = \rho_0(x).
\end{align}
This way, we can apply Pontryagin's Maximum Principle in a similar way to the deterministic optimal control problems. The result is the mean--field games system
\begin{align} 
&\rho(x,t) \frac{\delta \ell_d}{\delta u} - \rho(x,t) \nabla \lambda(x,t) = 0 \label{eq:opt_cond_mfg} \\
&\partial_t \rho(x,t) + \nabla \cdot \left( u(x,t) \rho(x,t) \right) - \frac{\nu^2}{2}\nabla \cdot (\nabla \rho(x,t)) = 0 \label{eq:fokker_planck} \\
&\partial_t \lambda(x,t) + u(x,t) \nabla \lambda(x,t) + \frac{\nu^2}{2} \nabla \cdot (\nabla \lambda(x,t) ) = \ell_d \label{eq:stoch_hjb}
\end{align}
where it is basically a coupled system of Fokker--Planck equation, \eqref{eq:fokker_planck}, and stochastic Hamilton--Jacobi--Bellman equation, \eqref{eq:stoch_hjb}. In \cite{G2021}, the network parameter $u$ is substituted by $u = \omega + \frac{\nu^2}{2}\nabla \mathrm{log}(\rho)$ and this is done to reduce the system of equations \eqref{eq:opt_cond_mfg}, \eqref{eq:fokker_planck} and \eqref{eq:stoch_hjb} to the following equation 
\begin{equation}\label{eq:mfg_euler_poincare}
\begin{aligned}
&\partial_t \frac{\delta \ell_d}{\delta \omega} + \mathrm{ad}^{\ast}_{\omega}\frac{\delta \ell_d}{\delta \omega} = \rho \nabla \frac{\delta \ell_d}{\delta \rho}\\
&\partial_t \rho + \nabla \cdot \left( \omega \rho \right) = 0
\end{aligned}
\end{equation}
where $\mathrm{ad}^{\ast}: \mathfrak{X}  \times \mathfrak{X}^{\ast} \to \mathfrak{X}^{\ast}$ is the coadjoint map and in this case it is defined in coordinates of $\mathbb{R}^N$ as
\begin{align}\label{eq:coadjoint_operator}
\left( \mathrm{ad}_\omega^{\ast} \frac{\delta \ell_d}{\delta \omega} \right)_i = \left( \partial_{x^j} \left( \left(\frac{\delta \ell_d}{\delta \omega}\right)_i \omega^j \right) + \left(\frac{\delta \ell_d}{\delta \omega}\right)_j \partial_{x^i}\omega^j \right)_{i}.
\end{align} 
One thing to note that the approach of \cite{G2021} relies on the Lagrangian formulation of mechanics and we cannot apply Hamilton--Jacobi theory directly, but in order to do so, we need to translate it into the Hamiltonian point of view. 

\begin{remark}
It is worth pointing out that when $\nu \to 0$ in \eqref{eq:mfg_euler_poincare}, it is in effect the celebrated hydrodynamics formulation of the Monge--Kantorovitch mass problem of optimal transport \cite{benamou2000computational}. And for the optimal control problem, with $\nu \to 0$ in \eqref{eq:opt_cond_mfg}, \eqref{eq:fokker_planck} and \eqref{eq:stoch_hjb}, it was shown that it reduces to a Euler's fluid equation in \cite{bloch2000optimal}.
\end{remark}
\begin{remark}
The right hand term that appears in \eqref{eq:mfg_euler_poincare}, the density times a potential of a scalar does appear, albeit with a different sign and a scalar function, when the metamorphosis equations of pattern matching are written as Euler-Poincar\'{e} equations \cite{metamorphosis}. Moreover, it is related to a similar system of partial differential equations known as the two--component Camassa--Holm equations \cite{chen2006two, Falqui_2005}
\end{remark}

\subsection{Deep Learning as a Lie-Poisson System}\label{sec:lp_form}
In this section, we formulate the reduced mean--field games \eqref{eq:mfg_euler_poincare} to a Lie--Poisson system, and in doing so, we translate it into the language of Lie--Poisson Hamilton--Jacobi theory. This allows us to construct a Poisson mapping from the target of the training dataset to the initial conditions.
\\

This first problem that we need to overcome is that the Hamiltonian for the mean--field games \eqref{eq:opt_cond_mfg}, \eqref{eq:fokker_planck} and \eqref{eq:stoch_hjb} is a covariant Hamiltonian and theorem \ref{prop:SemiLPHJ} is not adapted for such Hamiltonians. For covariant Hamiltonians, there has been attempt to introduce a general Hamilton--Jacobi theory, such as the work done in \cite{horava_covariant_1991, mclean_covariant_2000, von_rieth_hamilton-jacobi_1984}. An alternative approach in \cite{vankerschaver_generating_2013}, where boundary integrals of the variational principle are used to derive generating functions for partial differential equations. Having said that, the route taken here is to expand the work with a larger configuration manifold, that is, the group of pseudospectral symbols instead of the group diffeomorphism.  In order to use Lie--Poisson Hamilton--Jacobi theory, we discretise the group of pseudospectral symbols in space only. This is inspired by the work of \cite{pavlov2011structure} where the group of diffeomorphisms is discretised as $GL(N)$ and EPDiff equation is written in terms of vector fields on dual of the discretised Lie algebra which is a subalgebra of $\mathfrak{gl}(N)$. 
\\

For the deep learning problem, the Hamiltonian for the system \eqref{eq:mfg_euler_poincare} is
\begin{equation}\label{eq:hamiltonian_deep_0}
\begin{aligned}
H &= \frac{1}{2} \int_{Q} \rho \| \omega \|^2 d\mathrm{Vol}(x) + \frac{\nu^2}{2}\int_{Q}  \rho \left\langle \omega, \nabla \mathrm{log}(\rho) \right\rangle  d\mathrm{Vol}(x) \\
&+ \frac{\nu^4}{8} \int_{Q} \rho \|  \nabla \mathrm{log}(\rho)\|^2 d\mathrm{Vol}(x)
\end{aligned}
\end{equation}
and $H$ is a function defined on the dual of the Lie algebra of a semidirect product $\mathfrak{s}^{\ast} = \mathfrak{g}_{\mathrm{INT}}^{\ast}\oplus \mathrm{Dens}(Q)$. The corresponding group is $S = G_{\mathrm{INT}} \oplus \mathrm{Dens}(Q)$. From the definitions of group actions on itself, tangent and cotangent lifts defined in section \ref{sec:lphj_theory}, here $G_{\mathrm{INT}}$ acts on $\mathrm{Dens}(Q)$ by push-forward, i.e.
\begin{align}\label{eq:zeta}
    \zeta(g)\rho = (g)_{\ast}\rho, \quad g \in G_{\mathrm{INT}}
\end{align}
Another important map that we need to define is the algebra $\mathfrak{g}_{\mathrm{INT}}$ representation on $\mathrm{Dens}(Q)$, denoted by $\zeta': \mathfrak{g}_{\mathrm{INT}} \to \mathrm{End}\left( \mathrm{Dens}(Q) \right)$, is given by 
 \begin{align}
 \zeta'(U) \rho = - \pounds_{U} \rho,
 \end{align}
where $U \in \mathfrak{g}_{\mathrm{INT}}$ and $\pounds_{U}$ denotes the Lie derivative along the direction of $U$, and in coordinates $ \zeta'(U) \rho =  \mathrm{div}\left( \rho U \right)$. Once we have the representation on the Lie algebra, we derive the representation of the dual of the Lie algebra $\mathfrak{g}_{\mathrm{INT}}^{\ast}$, i.e. the mapping $\left( \zeta'_{U} \right)^{\ast} : \mathrm{Dens}(Q)^{\ast} \to \mathfrak{g}^{\ast}_{\mathrm{INT}}$. It is obtained by using the pairing and for any $\rho \in \mathrm{Dens}(Q)$, $\eta \in \mathrm{Dens}(Q)^{\ast}$ and $U \in \mathfrak{g}_{\mathrm{INT}}$ we have
\begin{equation}\label{eq:diamond_op}
\begin{aligned}
  \left\langle \eta \diamond \rho, U \right\rangle &= \left\langle \left( \zeta'_{\eta} \right)^{\ast}\rho, U \right\rangle =- \left\langle \eta, \zeta'_{U}\rho \right\rangle \\
  &= -\int \eta \nabla \cdot (U \rho) \, dx =-\int \rho \nabla \eta \cdot U \, dx ,
\end{aligned}
\end{equation}
and this means $\eta \diamond \rho = \left( \zeta'_{\eta} \right)^{\ast}\rho = \rho \nabla \eta$. The operator $\diamond: \mathrm{Dens}(Q) \times \mathrm{Dens}(Q)^{\ast} \to \mathfrak{g}_{\mathrm{INT}}^{\ast}$ is often used in describing motion of various hydrodynamics, introduced in \cite{holm_eulerpoincare_1998}, and it used to map Clebsch variables to the dual of Lie algebra  \cite{holm_poisson_1983}. 
\\

Going back to training of deep neural networks on Poisson manifolds and for this particular problem  we choose the symbol to be
\begin{align}\label{eq:symbol_diff}
    U(x,\xi) = u(x) \, \xi^{-1} \in \mathfrak{g}_{\mathrm{INT}},
\end{align}
and we require that the dual space $\mathfrak{g}_{\mathrm{INT}}^{\ast}$ should be the space of pseudodifferential symbols of order $2$, i.e.
\begin{align}
M(x,\xi) = \sum_{i=0}^2 m_i(x)\xi^i
\end{align}
The importance of these conditions on the Lie algebra is that for the functionals $F$ and $H$ substituting them into the Lie--Poisson bracket \eqref{eq:lie_poisson_pso}, we obtain
\begin{align}
    \left\{ F,H \right\} = \int m(x) \left[ \frac{\delta F}{\delta m}\nabla\left( \frac{\delta H}{\delta m}\right) - \frac{\delta H}{\delta m}\nabla\left( \frac{\delta F}{\delta m}\right)\right] \, dx.
\end{align}
The Lie--Poisson bracket of functionals on the dual of the Lie algebra of semidirect product $\mathfrak{s}^{\ast} = \mathfrak{g}_{\mathrm{INT}}^{\ast} \oplus \mathrm{Dens}(Q)^{\ast}$ is
\begin{align}\label{eq:Lie_Poisson_bracket_semi}
\{ F, H \}_{\pm}(M, \rho) = \pm \left\langle M, \left[ \frac{\delta F}{\delta M},  \frac{\delta H}{\delta M}\right] \right\rangle \mp \left\langle  \rho, \pounds_{\frac{\delta F}{\delta M}}\frac{\delta H}{\delta \rho} - \pounds_{\frac{\delta H}{\delta M}}\frac{\delta F}{\delta \rho}  \right\rangle,
\end{align}
where $F,H: \mathfrak{s}^{\ast} \to \mathbb{R}$ and the first term of the Lie--Poisson bracket is given by \eqref{eq:lie_poisson_pso}. In explicit terms,
\begin{equation}\label{eq:Lie_Poisson_bracket_explict}
\begin{aligned}
\{ F, H \}_{\pm}(M, \rho) &= \pm\int_{Q}m(x) \left[ \frac{\delta F}{\delta m}\nabla\left( \frac{\delta H}{\delta m}\right) - \frac{\delta H}{\delta m}\nabla\left( \frac{\delta F}{\delta m}\right)\right] \, dx \\
&\mp\int_{Q} \rho \left[ \left(\frac{\delta H}{\delta m} \cdot \nabla \right) \frac{\delta F}{\delta \rho} - \left(\frac{\delta F}{\delta m} \cdot \nabla \right) \frac{\delta H}{\delta \rho}\right] \, dx.
\end{aligned}
\end{equation}
\begin{remark}
This Lie--Poisson bracket is the same bracket for ideal compressible fluid \cite{marsden1983hamiltonian, marsden_semidirect_1984}.
\end{remark}
Since the Lie--Poisson bracket is in terms of $(M, \rho) \in \mathfrak{s}^{\ast}$, and that necessitates the need to write the Hamiltonian \eqref{eq:hamiltonian_deep_0} in those coordinates and for that we use $M= \rho \omega$,
\begin{equation}\label{eq:hamiltonian_lp_deep}
\begin{aligned}
H(M, \rho) &= \frac{1}{2} \int_{Q} \frac{\|M \|^2}{\rho} d\mathrm{Vol}(x) + \frac{\nu^2}{2}\int_{Q}   \left\langle M, \nabla \mathrm{log}(\rho) \right\rangle  d\mathrm{Vol}(x) \\
&+ \frac{\nu^4}{8} \int_{Q} \rho \|  \nabla \mathrm{log}(\rho)\|^2 d\mathrm{Vol}(x)
\end{aligned}
\end{equation}
Plugging in $F = (M, \rho)$, we obtain the Lie--Poisson on semidirect product space
\begin{equation}\label{eq:Lie_poisson_semi}
\begin{aligned}
\partial_t M &= \{M,H \}_{+}  \\
\partial_t \rho &= \{\rho ,H \}_{+},
\end{aligned}
\end{equation}
or in coordinates od $\mathbb{R}^N$, it is written as
\begin{equation}\label{eq:Lie_poisson_semi2}
\begin{aligned}
\partial_t m_i &= m_i\partial_{x^i} \frac{\delta H}{\delta m_i} + \partial_{x^i}\left(m_i \frac{\delta H}{\delta m_i} \right) + \rho\partial_{x^i} \frac{\delta H}{\delta \rho}, & \text{for} \, i=1,\dots, N\\
\partial_t \rho &= - \partial_{x^j}\left( \rho \frac{\delta H}{\delta m_j}\right),
\end{aligned}
\end{equation}
where $\mathrm{ad}^{\ast} : \mathfrak{g}_{\mathrm{INT}} \times \mathfrak{g}_{\mathrm{INT}}^{\ast} \to \mathfrak{g}_{\mathrm{INT}}^{\ast}$ and it is the same as in \eqref{eq:coadjoint_operator}. 
\\

\begin{proposition}
The deep learning dynamics can be written in Lie--Poisson form
\begin{equation}\label{eq:Lie_poisson_semi3}
\begin{aligned}
\partial_t M &= \pm \mathrm{ad}^{\ast}_{\frac{\delta H}{\delta M}} M \mp \frac{\delta H}{\delta \rho} \diamond \rho \\
\partial_t \rho &= \mp \pounds_{\frac{\delta H}{\delta M}} \rho,
\end{aligned}
\end{equation}
where $\mathrm{ad}^{\ast} : \mathfrak{g}_{\mathrm{INT}} \times \mathfrak{g}_{\mathrm{INT}}^{\ast} \to \mathfrak{g}_{\mathrm{INT}}^{\ast}$ and it is the same as in \eqref{eq:coadjoint_operator}. 
\end{proposition}
\begin{proof}
Let $(\eta, \varrho), (u, a) \in \mathfrak{s}$ and $(M, \rho) \in \mathfrak{s}^{\ast}$, from the definition of the coadjoint map and the diamond operator, we then have
\begin{align*}
\left\langle \left(  \mathrm{ad}^{\ast}_{u}M - a \diamond \rho,- \pounds_{u}\rho \right), \left( \eta, \varrho \right) \right\rangle & = \left\langle  \mathrm{ad}^{\ast}_{u}M - a \diamond \rho, \eta \right\rangle + \left\langle- \pounds_{u}\rho , \varrho \right\rangle \\
&= \left\langle M, \mathrm{ad}_u \eta  \right\rangle + \left\langle \rho , \pounds_{u}\varrho - \pounds_{\eta}a \right\rangle \\
&= \left\langle M, [u, \eta]  \right\rangle + \left\langle \rho , \pounds_{u}\varrho - \pounds_{\eta}a \right\rangle \\
&= \left\langle M, [u, \eta]  \right\rangle - \left\langle \rho \nabla a, \eta \right\rangle + \left\langle \rho , \pounds_{u}\varrho  \right\rangle \\
&= \left\langle \{M, u \},\eta  \right\rangle + \left\langle \pounds_{u} \rho , \varrho  \right\rangle \\
&= \left\langle \left(\{M, u \}, \{\rho, u \}\right), \left( \eta, \varrho \right)  \right\rangle.
\end{align*}
From the equation we get
\begin{align*}
    \left(  \mathrm{ad}^{\ast}_{u}M - a \diamond \rho,- \pounds_{u}\rho \right) = \left(\{M, u \}, \{\rho, u \}\right)
\end{align*}
and thus we can write equation \eqref{eq:Lie_poisson_semi} as
\begin{align*}
\partial_t M &= \pm \mathrm{ad}^{\ast}_{\frac{\delta H}{\delta M}} M \mp \frac{\delta H}{\delta \rho} \diamond \rho \\
\partial_t \rho &= \mp \pounds_{\frac{\delta H}{\delta M}} \rho.
\end{align*}
\end{proof}

For the time being, we only consider the case $Q$ is $\mathbb{R}^n$, in order for the definitions of $\Psi \mathrm{DS}$ to hold. Now, we can apply the Lie--Poisson Hamilton--Jacobi theory for semidirect products as
\begin{equation}
\begin{aligned}
\partial_t S_L((A,\rho)) &+ H\left( TR^{\ast}_{A}\mathbf{d}_{A}S_L + (\zeta'_{\rho})^{\ast} \mathbf{d}_{\rho} S_L , \mathbf{d}_{\rho}S_L \right)  = 0 \\
\end{aligned}
\end{equation}
The solution of equation \eqref{eq:semiLPHJeqn} generates a Poisson map that defines the flow of the reduced Hamiltonian $H$. The Poisson map transforms the initial momentum map $M_0 \to M(t)$, where
\begin{equation}\label{eq:Pi_0_mfg}
\begin{aligned}
(M_0, \rho_0) =& (- \mathbf{d}_A S_L \circ A, \zeta(A)^{\ast} \mathbf{d}_{\rho} S_L),\\
\text{and} \quad (M(t), \rho(t)) =& \left( A \circ \mathbf{d}_{A}S_L  + (\zeta'_{\rho})^{\ast} \mathbf{d}_{\rho} S_L  , \mathbf{d}_{\rho}S_L \right).
\end{aligned}
\end{equation}
Now we decipher the two previous equations in \eqref{eq:Pi_0_mfg}. The first component $\mathbf{d}_A S_L \circ A$ is the multiplication in $\Psi \mathrm{DS}$ and the second, as mentioned, is the action of $\Psi G$ acting on $\mathrm{Dens}(Q)$ by pullback. The term $(\zeta'_u)^{\ast} \mathbf{d}_u S_L$ is the diamond operator defined in \eqref{eq:diamond_op}.   The terms $(\mathbf{d}_A S_L, \mathbf{d}_u S_L)$ can be eliminated from the second equation in \eqref{eq:Pi_0_mfg}, by taking the left action of $(A, u)^{-1}$ of $(M_0, \rho_0)$. This results in
\begin{equation}\label{eq:coadjoint_action_0}
\begin{aligned}
    (M(t), \rho(t)) &= (- A \circ M_0 \circ A^{-1} + (\zeta'_{\rho})^{\ast} \rho_0(A \circ x), \rho_0(A \circ x) )\\
    &= \mathrm{Ad}_{(A,u)^{-1}}^{\ast} (M_0, \rho_0),
\end{aligned}
\end{equation}
where $\mathrm{Ad}_{(A,u)^{-1}}^{\ast} : S \times \mathfrak{s}^{\ast} \to \mathfrak{s}^{\ast}$ is the coadjoint operator. 
\\

Despite the Lie--Poisson Hamilton--Jacobi theory introduces a level of complexity to the matter; the power of the method comes out when it is employed for constructing numerical algorithms. Up until now, we have worked exclusively with continuous-time representation of deep learning, but for it to applicable we need to discretise it.  The discrete Lie--Poisson Hamilton--Jacobi theory would yield an $N$ dimensional system that preserves the geometric structure of the continuous equations. Furthermore, the flow or solution trajectories, generated by the Hamiltonian vector-field corresponding to \eqref{eq:Lie_Poisson_bracket_explict} belong to the coadjoint orbit. \emph{The coadjoint orbit}, $\mathcal{O}_{(M, \rho)}$, is the set of elements of $\mathfrak{s}^{\ast}$ that are moved by the elements of the group $S = \Psi G \times\mathrm{Dens}(Q)$, or mathematically,
\begin{align}
    \mathcal{O}_{(M, \rho)} := \left\{ \mathrm{Ad}^{\ast}_{(A, u)^{-1}}(M, \rho) \, | \, (A, u) \in \Psi G \right\}, 
\end{align}
and $(M, \rho) \in \mathfrak{s}^{\ast}$. The  coadjoint orbit is in fact a sub-manifold of $\mathfrak{s}^{\ast}$ and it is endowed with a symplectic structure \cite{marsden_coadjoint_1983}. The Lie--Poisson Hamilton--Jacobi constructs a numerical solution that naively belongs to $\mathcal{O}_{(M, \rho)}$ and this makes it worth studying. The preservation of the coadjoint orbit was further looked at in \cite{Marsden_1999}. 
\\

What is interesting regarding the discrete equations is that they are also a Hamiltonian system and this is a consequence of applying reduction to generating function on cotangent bundle.  In \cite{G2020} it was shown how to derive the modified Hamiltonian for the discrete equations obtained from Lie--Poisson Hamilton--Jacobi theory, albeit the systems considered in \cite{G2020} are stochastic. However, the same steps can be followed for the deterministic case.
\\

The fact that mean--field games system \eqref{eq:opt_cond_mfg}, \eqref{eq:fokker_planck} and \eqref{eq:stoch_hjb} is written as a coadjoint motion, then according to theorem $2.7$ in \cite{gay-balmaz_clebsch_2011}, the optimal solution is restricted to $\mathcal{O}_{(M, \rho)}$. This means $(M_0, \rho_0)$ as well as $(M_T, \rho_T)$ are on the same coadjoint orbit $\mathcal{O}_{(M, \rho)}$ and that means the curve $(M(t), \rho(t)) \in \mathfrak{s}^{\ast}$ is a minimising curve of the functional and it also connects the two densities $\rho_0$ and $\rho_T$.

\section{Hamilton--Jacobi Integrators for Deep Learning}\label{sec:hj_for_dl}

In this section we show the application of generating function in constructing a structure--preserving integrator for training deep neural network. Symplectic transformations, that are generated by the solution of the Hamilton--Jacobi equation, preserve the structure and it is the basis for construction of the numerical scheme. This is done by applying a numerical method for solving the Hamilton--Jacobi equation and then use that with \eqref{eq:symplecto} to obtain the approximated solution. From Pontryagin's Maximum Principle, the necessary conditions for optimality of the solution of an optimal control problem satisfy a Hamiltonian system \eqref{eq:Ham_equn_sec2}. Thus, it is fitting that a geometric integrator is used to solve optimal control problem to guarantee that the numerical solution is representative of the control system. First, we describe how optimal is used to formulate the training of deep neural network in section \ref{sec:opt_control_deep} and, here, we adhere to the approach of \cite{li2017maximum, benning2019deep, celledoni_structure_2020}. The resultant Hamilton's equations are then solved numerically using structure--preserving integrator. 
\\

The idea of using generating functions for numerical integrators was first introduced in \cite{kang_construction_1989, channell_symplectic_1990}. Recall that the solution of a Hamiltonian system is a symplectomorphism and one time-step integration is also a symplectic transformation $(q^k, p^k) \to (q^{k+1}, p^{k+1})$ generated by the solution of Hamilton--Jacobi equation at time $\Delta t$. Using a generating function of type--II and \eqref{eq:symplecto}, we obtain the following numerical scheme
\begin{equation}\label{eq:scheme_gen_func}
\begin{aligned}
q^{k+1} = q^k + \frac{\partial S_2}{\partial p}(q^k, p^{k+1}, \Delta t), \quad p^{k+1} = p^k - \frac{\partial S_2}{\partial q}(q^k, p^{k+1}, \Delta t)
\end{aligned}
\end{equation}
and such scheme preserves the symplectic form $\omega = \mathbf{d}q^{k} \wedge \mathbf{d} p^{k} = \mathbf{d}q^{k+1} \wedge \mathbf{d} p^{k+1}$. The generating function, $S_2$ is the approximation of solution of Hamilton--Jacobi equation and it can be expressed as a truncated power series
\begin{align}\label{eq:power_s2}
S_2(q, p, t) = \sum\limits_{i=1}^{m}t^i S_{2,i}(q, p, t), 
\end{align}
where the coefficients of the power series are functions in term of the Hamiltonian $H$ and its partial derivatives
\begin{align*}
S_{2,1} = H(q,p), \quad S_{2,2} = \frac{1}{2} \left\langle \frac{\partial H}{\partial q}, \frac{\partial H}{\partial p} \right\rangle, \quad S_2,3 =  \dots
\end{align*}
The order of scheme \eqref{eq:scheme_gen_func} is determined by the number of terms, $m$, in the approximation \eqref{eq:power_s2}.  
\begin{remark}
When $m = 1$, i.e. $S = \Delta t H$, the scheme \eqref{eq:scheme_gen_func} becomes the symplectic Euler scheme
\begin{align*}
q^{k+1} = q^k + \Delta t \frac{\partial H}{\partial p}(q^k, p^{k+1}, \Delta t), \quad p^{k+1} = p^k - \Delta t \frac{\partial H}{\partial q}(q^k, p^{k+1}, \Delta t)
\end{align*}
\end{remark}

The application of structure--preserving integrators based on generating function can be used to approximate the solution of optimal control problem, now we use that to solve problem ~\ref{prob:opt_training_det}. Using Pontryagin's Maximum Principle, the control Hamiltonian $H: T^{\ast}Q\times \mathcal{U} \to \mathbb{R}$ is 
\begin{align}\label{eq:opt_control_Ham_resnet}
H(q_0, \dots, q_N, p_0, \dots, p_N, \theta) = \sum\limits_{i=0}^{N}\left\langle p_i, \sigma(q_i, \theta) \right\rangle - \frac{\gamma}{2} \| \theta \|^2
\end{align}
where $\gamma \in \mathbb{R}_{\geq 0}$ and the last term is the regularisation term. We also require that
\begin{align*}
\frac{\delta H}{\delta \theta} = \frac{\delta H}{\delta (u, b)}= 0,
\end{align*}
that is
\begin{align*}
\frac{1}{\gamma}\sum\limits_{i=0}^{N} p_i\frac{\partial \sigma}{\partial z}q_i^T - u &= 0 \\
 \frac{1}{\gamma}\sum\limits_{i=0}^{N} p_i\frac{\partial \sigma}{\partial z} - b &= 0
\end{align*}
and substitute the expression for $\theta$ into $H$, we obtain a Hamiltonian in terms of canonical coordinates $(q,p)$ on cotangent bundle $T^{\ast}Q$. Then the Hamiltonian,  $H: T^{\ast}Q \to \mathbb{R}$, becomes a function defined on the cotangent space rather than the Pontryagin's bundle and it is given by
\begin{equation}\label{eq:opt_control_Ham_resnet_cotq}
\begin{aligned}
H(q_0, \dots, q_N, p_0, \dots, p_N, t)  &= \sum\limits_{i=0}^{N}\left\langle p_i,  \sigma \left(q_i,  \left( \frac{1}{\gamma}\sum\limits_{j=0}^{N} p_j\frac{\partial \sigma}{\partial z}q_j^T, \frac{1}{\gamma}\sum\limits_{i=0}^{N} p_i\frac{\partial \sigma}{\partial z}   \right) \right) \right\rangle \\
&- \frac{1}{2\gamma} \left\| \sum\limits_{j=0}^{N} p_j\frac{\partial \sigma}{\partial z}q_j^T \right\|^2 - \frac{1}{2 \gamma} \left\| \sum\limits_{j=0}^{N} p_j\frac{\partial \sigma}{\partial z} \right\|^2
\end{aligned}
\end{equation}
\begin{figure}[h]
    \centering
         \begin{subfigure}[b]{0.47\textwidth}
         \centering
         \includegraphics[width=\textwidth]{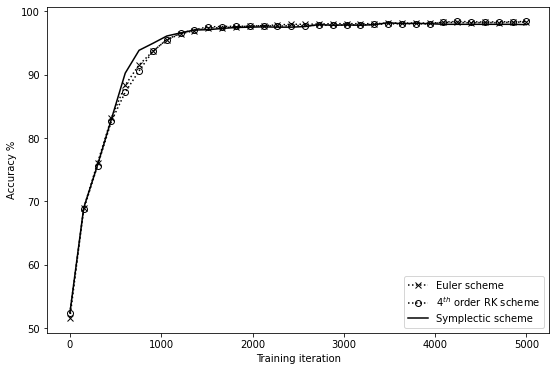}
         \caption{ }
         \label{fig:accuracy}
     \end{subfigure}
          \begin{subfigure}[b]{0.47\textwidth}
         \centering
         \includegraphics[width=\textwidth]{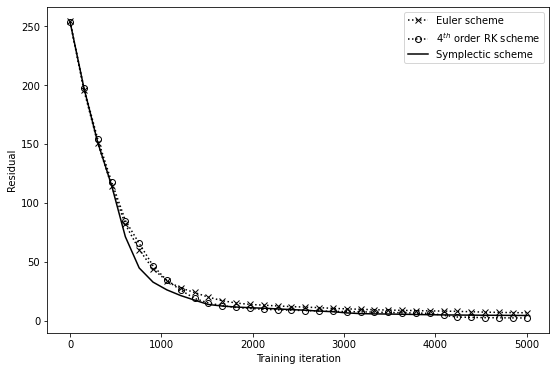}
         \caption{ }
         \label{fig:residual}
     \end{subfigure}
\hfill         
\begin{subfigure}[b]{0.47\textwidth}
         \centering
         \includegraphics[width=\textwidth]{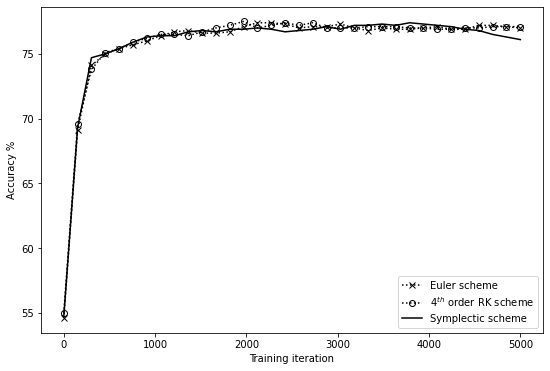}
         \caption{ }
         \label{fig:accuracy_d}
     \end{subfigure}
          \begin{subfigure}[b]{0.47\textwidth}
         \centering
         \includegraphics[width=\textwidth]{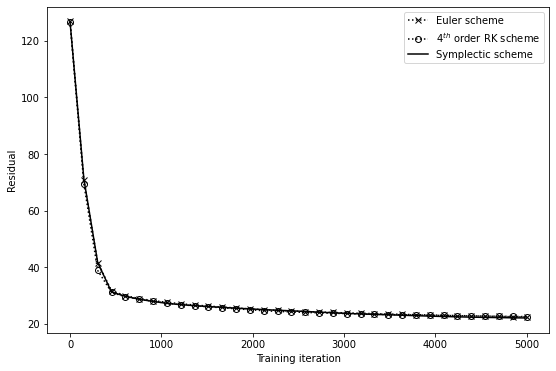}
         \caption{ }
         \label{fig:residual_d}
     \end{subfigure}
        \caption{The accuracy of the neural networks predictions is shown in figures ~\ref{fig:accuracy} and ~\ref{fig:accuracy_d}, for the two spirals and two circles datasets respectively, over $5000$ training iterations. On the other hand, the residual, computed by $\sum\limits_{i=0}^N \| \pi(q^(50)_i) - c^i \|^2$, is shown in figures ~\ref{fig:residual} and ~\ref{fig:residual_d}, for the two spirals and two circles datasets respectively.}
        ~\label{fig:plots_acc_residual}
\end{figure}
\begin{figure}[h!]
    \centering
         \begin{subfigure}[b]{0.95\textwidth}
         \centering
         \includegraphics[width=\textwidth]{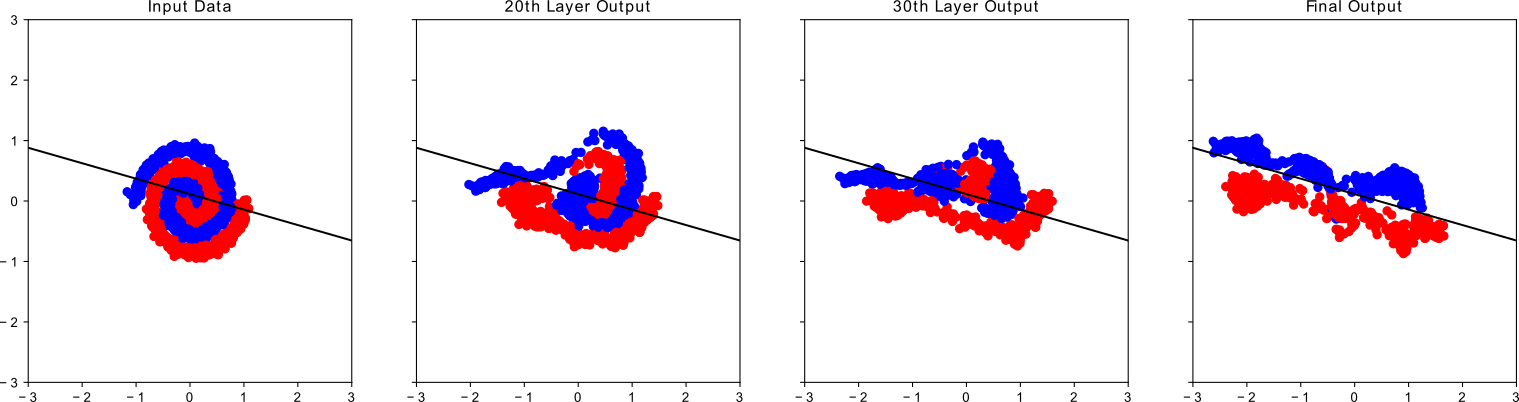}
         \caption{ }
         \label{fig:euler_snap_all}
     \end{subfigure}
     \hfill
          \begin{subfigure}[b]{0.95\textwidth}
         \centering
         \includegraphics[width=\textwidth]{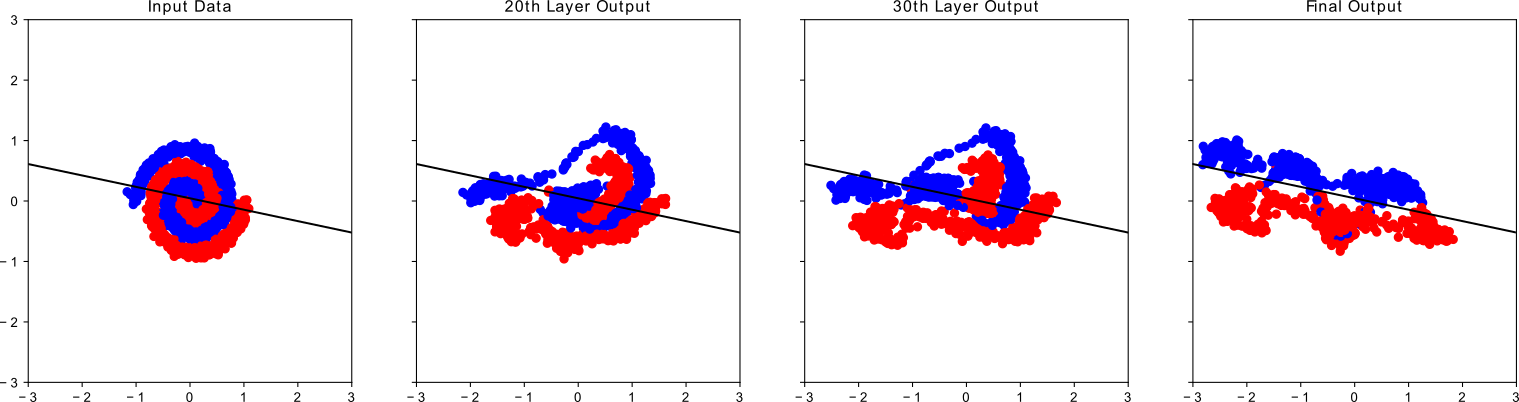}
         \caption{ }
         \label{fig:kutta_snap_all}
     \end{subfigure}
 \hfill     
 \begin{subfigure}[b]{0.95\textwidth}
         \centering
         \includegraphics[width=\textwidth]{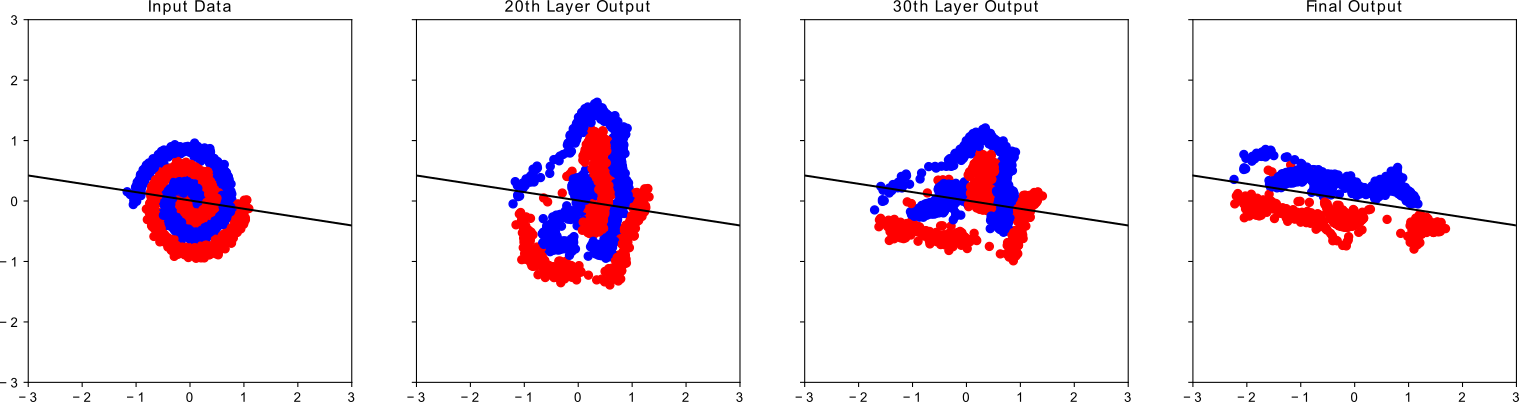}
         \caption{ }
         \label{fig:symplectic_snap_all}
     \end{subfigure}
        \caption{The transformation of two spirals training set $ \{(q_0^{(i)}, c^{(i)}_{N_t}) \}_{i = 0, \dots, N}$ with $N = 2000$ at the $20^{th}$ layer (the second column from the left), $30^{th}$ layer (the third column from the left), and the final layer (rightmost column). Figure \ref{fig:euler_snap_all}, shows the network discretised using Euler scheme, while figure ~\ref{fig:kutta_snap_all}, shows the results obtained using $4^{th}$ order Runge--Kutta method . Finally,  figure ~\ref{fig:symplectic_snap_all} is when the symplectic method using Euler scheme for the generating function. The solid black line separates between the two classes/categories. }
        ~\label{fig:unfold}
\end{figure}
\begin{figure}[h!]
    \centering
         \begin{subfigure}[b]{0.49\textwidth}
         \centering
         \includegraphics[width=\textwidth]{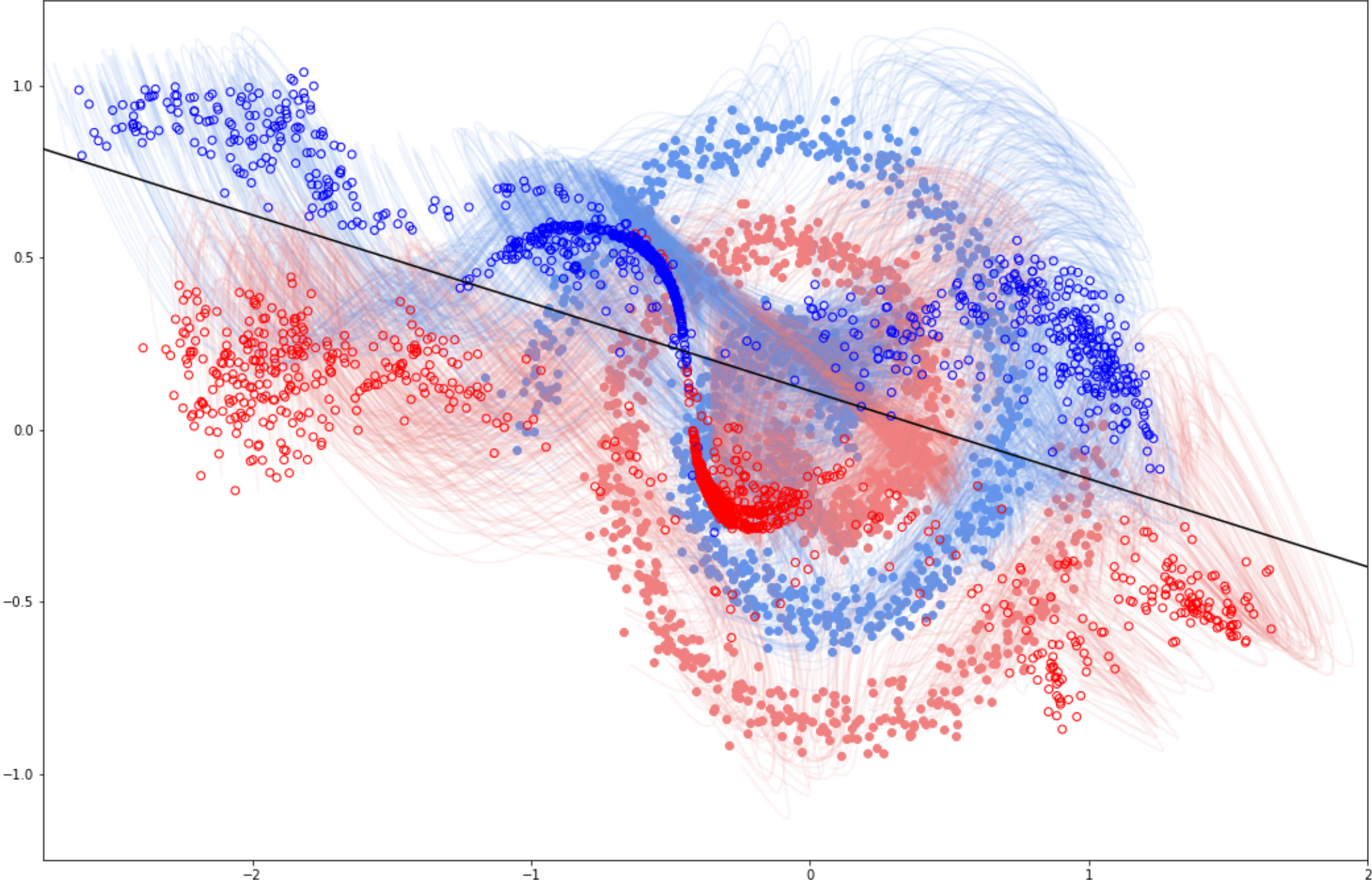}
         \caption{ }
         \label{fig:euler_paths}
     \end{subfigure}
     \hfill
          \begin{subfigure}[b]{0.49\textwidth}
         \centering
         \includegraphics[width=\textwidth]{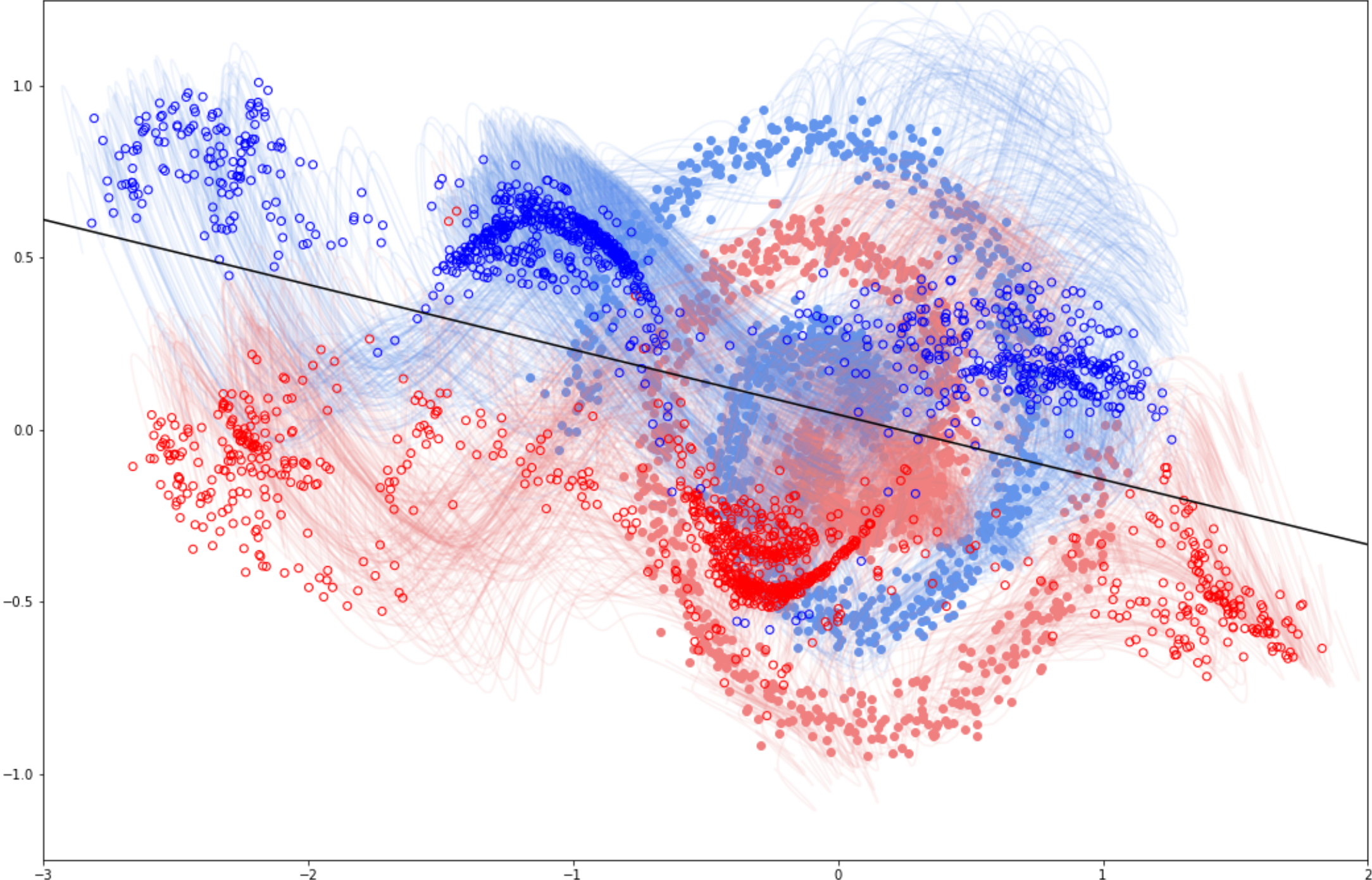}
         \caption{ }
         \label{fig:kutta_paths}
     \end{subfigure}
 \hfill     
 \begin{subfigure}[b]{0.49\textwidth}
         \centering
         \includegraphics[width=\textwidth]{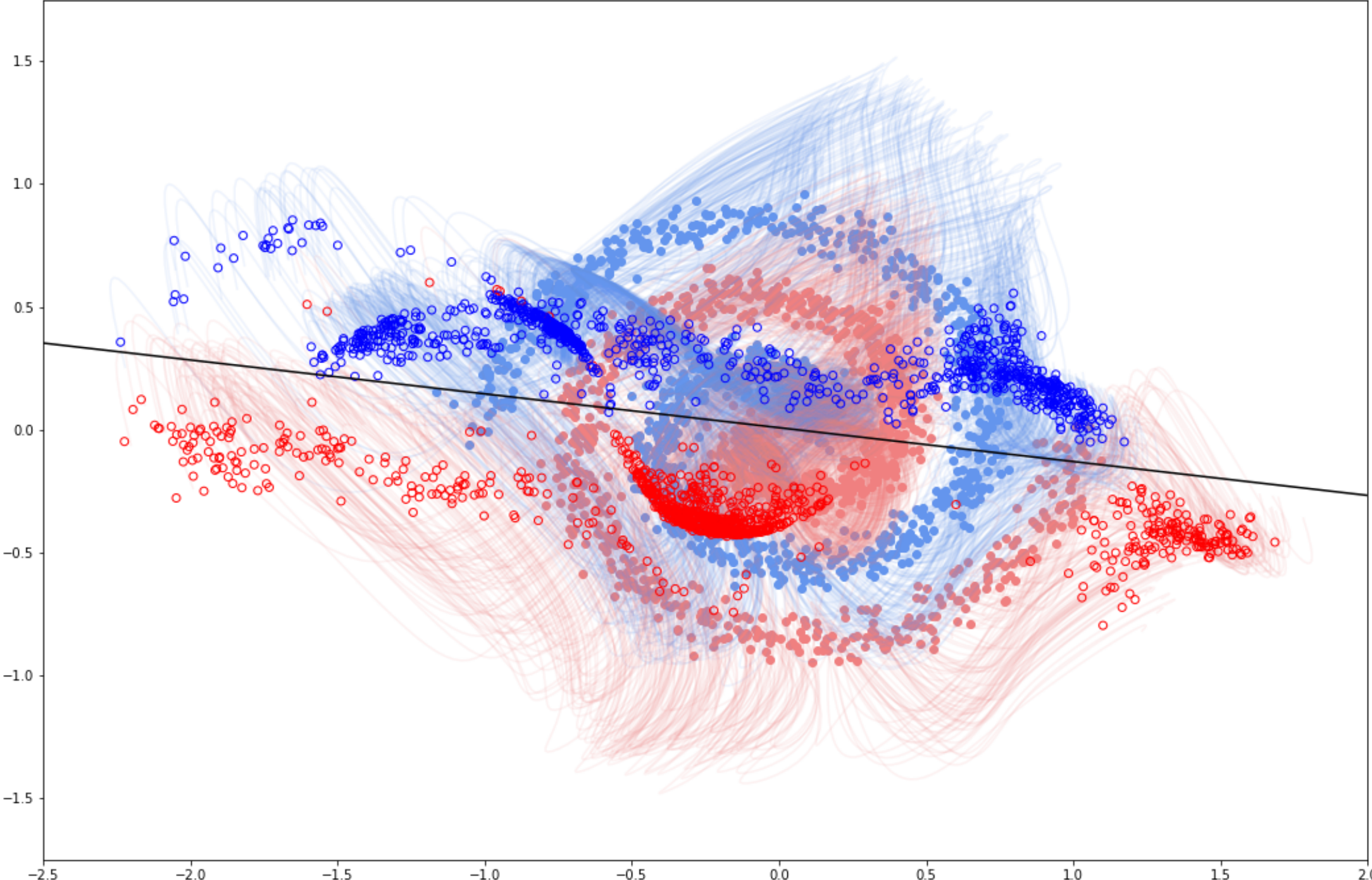}
         \caption{ }
         \label{fig:symplectic_paths}
     \end{subfigure}
        \caption{The trajectories of the two spirals data are plotted, where the solid dots are the initial data and the hollow circles are the final data. The light coloured lines are the individual paths each data point followed when iterated through the discretised continuous residual network. Figure ~\ref{fig:euler_paths} shows the paths when the network is discretised using Euler scheme, while figure ~\ref{fig:kutta_paths} and ~\ref{fig:symplectic_paths} show the trajectories of the $4^{th}$ order Runge--Kutta and symplectic integrator, respectively.  }
        \label{fig:trajectories}
\end{figure}
\begin{problem}\label{prob:opt_training_discrete}
Given the training set  $ \{(q_0^{(i)}, c^{(i)}_{N_t}) \}_{i = 0, \dots, N}$, compute the sequence $(q^k, p^k)$ such that it satisfies
\begin{equation}\label{eq:symplectic_resNet}
\begin{aligned}
q_i^{k+1} &= q_i^k + \Delta t \sigma \left( \frac{1}{\gamma} \sum\limits_{j=0}^{N} \frac{\partial \sigma}{\partial z} p_j^{k+1}  (q_j^k)^Tq_i^k +\frac{1}{\gamma}\sum\limits_{j=0}^{N} \frac{\partial \sigma}{\partial z} p_j^{k+1}  \right), \\
 p_i^{k+1} &= p_i^n - \frac{\Delta t}{\gamma} \frac{\partial \sigma}{\partial z} p_i^{k+1} \left( \sum\limits_{j=0}^{N} \frac{\partial \sigma}{\partial z} p_j^{k+1}  (q_j^k)^T \right)
\end{aligned}
\end{equation}
for $k=0, \dots, N_l-1$ and the boundary conditions $q^0 = q_0^{(i)}$ and $\pi(q^{N_t}) = c^{(i)}$ for each $i = 0, \dots, N$. 
\end{problem}
In order to iterate through \eqref{eq:symplectic_resNet}, we require $p^0$, however at this point, it is not available. Further, the value of $p^0$ must ensure that $q^0 = q_0^{(i)}$ and $\pi(q^{N_t}) = c^{(i)}$. There are a number of ways in which one can choose a suitable initial value for the costate equation and one of them is shooting method \cite{kraft_converting_1985}. The idea is to pick an arbitrary $p^0$ and then evaluate how different $\pi(q^{N_t})$ from   $c^{(i)}$ and use this error to modify $p^0$. The process is then is iterated to refine the choice of $p^0$. An alternative approach is to use the fact that we can derive the $p^{N_l}$ from the cost functional. Once we have that, we can use that to define \eqref{eq:symplectic_resNet} as the  Karush--Kuhn--Tucker condition for nonlinear programming. For such constraint optimisation problems, iterative methods such as sequential quadratic programming \cite{han_superlinearly_1976}, or even using Newton's iterative method. 
\\

To evaluate the numerical scheme presented in this section, we train the network for two separate classification problems: two spirals, and two concentric circles classification. That said, this is just a toy example as a sanity check. Both datasets contain data from only two classes, that means for all classification space $\mathcal{C} = \{0, 1\}$. The two spirals dataset contains $2000$ data points for training and $2000$ other data points for testing. Whereas, the two concentric circles dataset uses, for each training and testing, $1000$ data points.  The number of layers used for both datasets is $N_t = 50$ and $\Delta t = 0.075$. The training was done over $5000$ iterations. 
\\

In order to put the structure--preserving scheme into perspective, other numerical schemes for ordinary differential equations are used for neural networks: Euler method, and fourth order explicit Runge--Kutta method. The methods are compared using the error function in \eqref{eq:control_functional_det}, $L = \sum\limits_{i=0}^N \| \pi((q(50))^{(i)}) - c^{(i)} \|^2$,  and the accuracy at each training steps. The accuracy and residuals for the two datasets are shown in figure ~\ref{fig:plots_acc_residual}. Figure ~\ref{fig:unfold} shows snap shots of the data, only for the two spirals dataset, propagating through the layers of each discretisation of the continuous ResNet. At the final layer, we see that, for all implementations, the majority of data points of each class are carried in a way they become sorted. Figure ~\ref{fig:trajectories}, shows the trajectories that each data point takes for all three discretisations.

\subsection{Deep Neural Networks as Discrete Lie-Poisson System}\label{sec:Discrete_LP}

Before discretising \eqref{eq:Pi_0_mfg}, we first need to review a particular discretisation compatible with the Lie--Poisson Hamilton--Jacobi theory. The discretisation is done on the algebra of pseudodifferential symbols $\Psi \mathrm{DS}$ instead of discretising the Hamilton--Jacobi equation. The philosophy here is similar to \cite{mullen2009energy, pavlov2011structure}, where the volume preserving group of diffeomorphisms is discretised and approximated as a subgroup of $GL(N)$ Lie group instead. 
\\

Let $Q$ denote a smooth manifold, its discrete analogue is denoted by
\begin{align*}
\mathsf{Q} := \left\{ (x_1^i, x_2^i, \dots, x_N^i) \, : \, i=0, \dots, N_t, \quad (x_1^i, x_2^i, \dots, x_N^i) \in Q \right\},
\end{align*}
and in this discretisation, the space between each point in each dimension is a constant $\mathsf{D}_x$. Consider the pseudodifferential operator \eqref{eq:pso_operator}, when applied to a function $u  \in C^{\alpha}\left( Q \right)$, and instead of evaluating the Fourier transform on $Q$, it is evaluated on $\mathsf{Q}$. With out the loss of generality, in $1$ dimensions, the pseudodifferential operator on $u$ can be written as 
\begin{align}\label{eq:discretepsudo}
P_d u(x_j) = \frac{1}{(2 \pi)^n}\sum\limits_{j,k}\int e^{i(x_k- x_j) \xi} a_k(x_j) \xi^k u_j d\xi.
\end{align}
From the previous expression, the following polynomial
\begin{align}
p(x_j, \xi) = \sum\limits_{k \leq m} a_{k,j} \xi^k,
\end{align}
where $ a_{k,j} = a_k(x_j)$, is called \emph{discrete symbol of order $\mathbf{m}$}. We denote the algebra of pseudodifferential symbols of order $m$ on $\mathsf{Q}$ by $\Psi \mathrm{DS}_{d, m}$. Basic operations and definitions carry also to pseudodifferential operators on $\mathsf{Q}$. 
\begin{definition}\label{def:mul_discrete}
Let $A \in \Psi \mathrm{DS}_{d, m}$ and $B \in \Psi \mathrm{DS}_{d, n}$, then the multiplication of $A \circ B$ on $\mathsf{Q}$ is the operator $C$ with symbol of order $m+n$ defined by
\begin{align}\label{eq:mul_discrete}
    c_j(\xi) =& \sum\limits_{|\alpha| < N} \frac{1}{\alpha!} \partial^{\alpha}_{\xi}a_j(\xi)\mathsf{D}_x^{\alpha}{b}_j( \xi)
\end{align}
where $\mathsf{D}_x^{\alpha}$ is the discrete difference operator in the $x$ direction of order $\alpha$.
\end{definition}
\begin{definition}
Let $A \in \Psi \mathrm{DS}_{d, m}$ and $B \in \Psi \mathrm{DS}_{d, n}$, the commutator bracket for the $\Psi \mathrm{DS}_{d}$ is
\begin{align}
    \left[ A, B \right](x_j) = \sum_{n \geq 1} \frac{1}{\alpha!} \partial^{\alpha}_{\xi}a_j(\xi)\mathsf{D}_x^{\alpha}{b}_j( \xi) - \mathsf{D}_x^{\alpha}{a}_j( \xi) \partial^{\alpha}_{\xi}b_j(\xi)
\end{align}
\end{definition}
Another important map that is used extensively is the trace map $\mathrm{Tr}: \Psi \mathrm{DS}_d \to \mathbb{R}$, and here for  $A \in \Psi \mathrm{DS}_{d, m}$, it is defined by
\begin{align}
    \mathrm{Tr}\left( A \right) = \sum\limits_{ j} a_{-1,j}.
\end{align}
Having defined the trace mapping, this enable us to define the pairing between the algebra $\Psi \mathrm{DS}_{d}$ and its dual $\Psi \mathrm{DS}_{d}^{\ast}$ defined as a map  $\langle \cdot, \cdot \rangle : \Psi \mathrm{DS}_{d} \times \Psi \mathrm{DS}^{\ast}_{d} \to \mathbb{R}$ described by
\begin{align} \label{eq:pairing_discrete}
    \left\langle A, B \right\rangle = \mathrm{Tr}(A \circ B),
\end{align}
where $A \in \Psi \mathrm{DS}_{d}$, and  $B \in \Psi \mathrm{DS}^{\ast}_{d}$.
\begin{remark}
This is useful when the the space $\Psi \mathrm{DS}_d$ is decomposed into two subspaces as in \eqref{eq:pso_decompose} to help identify the the Lie algebra with its dual.
\end{remark}

To directly apply Lie--Poisson Hamilton--Jacobi theory of proposition \ref{prop:SemiLPHJ}, we first need to give the explicit expression for the group action and its cotangent lift on pseudodifferential algebra on $\mathsf{Q}$.  As we are only interested in a subalgebra of  $\Psi \mathrm{DS}_d$, we only state the algebraic expressions of the subalgebra of interest. From this point onwards, we use the notation $\mathfrak{g}_{\mathrm{INT}}(\mathsf{Q})$ to denote the algebra of pseudodifferential integral operators on the discrete set $\mathsf{Q}$ and $\mathfrak{g}_{\mathrm{INT}}^{\ast}(\mathsf{Q})$ its dual. Analogous to the continuous case, the Lie algebra $\mathfrak{g}_{\mathrm{INT}}(\mathsf{Q})$ has an associated Lie group $G_{\mathrm{INT}}(\mathsf{Q})$ and elements of this Lie group, $g_j(s)$, satisfy the following differential equation
\begin{align*}
\frac{d}{ds} g_j(s) = A_j \circ g_j(s),
\end{align*}
where $A_j \in\mathfrak{g}_{\mathrm{INT}}(\mathsf{Q})$ and $g_j(0) = \mathrm{Id}$. The subscript $j$ denotes that action is evaluated at $x_j \in \mathsf{Q}$. 
\\

Computing the tangent and cotangent lifts the actions of $G_{\mathrm{INT}}(\mathsf{Q})$ on  $\mathfrak{g}_{\mathrm{INT}}(\mathsf{Q})$ and $\mathfrak{g}^{\ast}_{\mathrm{INT}}(\mathsf{Q})$ for all orders is beyond the scope of this project. However, we only consider the order $-1$  of pseudodifferential symbols, because it is the symbol that would yield the desired Lie--Poisson bracket \eqref{eq:Lie_Poisson_bracket_explict}.
\begin{definition}
Consider the discrete symbol $P_j = p_{1,j} \xi^{-1} +  p_{2,j} \xi^{-2} +  p_{3,j} \xi^{-3} \in \mathfrak{g}_{\mathrm{INT}}(\mathsf{Q})$ and the group action $g_j(s) =  1 + u_{1,j}(s) \xi^{-1} +u_{2,j}(s) \xi^{-2}  +u_{3,j}(s) \xi^{-3} $ and its tangent lift of the left action on $P_j$ is
\begin{align*}
    TL_{g_j}P_j &=p_{1,j}\xi^{-1} + p_{2,j}\xi^{-2} + \left( u_{1,j}p_{1,j} + p_{3,j} \right) \xi^{-3} \\
    &+ \left( u_{1,j}p_{2,j} - 2u_{1,j}\mathsf{D}_x p_{1,j} + u_{2,j}p_{1,j} \right) \xi^{-4} + \mathcal{O}(\xi^{-5}),
\end{align*}
and the tangent lift of the right action on $P_j$ is
\begin{align*}
    TR_{g_j}P_j &= p_{1,j}\xi^{-1}+ p_{2,j}\xi^{-2} + \left(u_{1,j}p_{1,j}  + p_{3,j} \right)\xi^{-3} \\
    &+ \left(u_{1,j}p_{2,j} + u_{2,j}p_{1,j} - p_{1,j}\mathsf{D}_x u_{1,j} \right)\xi^{-4} + \mathcal{O}(\xi^{-5}) .
\end{align*}
\end{definition}
Meanwhile,  the cotangent lift of the action of $G_{\mathrm{INT}}(\mathsf{Q})$ on  $\mathfrak{g}^{\ast}_{\mathrm{INT}}(\mathsf{Q})$ are computed according to \eqref{eq:left_cotangent_lift} and \eqref{eq:left_cotangent_lift}. 
\begin{definition}
Consider the discrete symbol $\alpha_j =  \alpha_{1,j} \xi + \alpha_{2,j} \xi^2$  on the dual space of $\mathfrak{g}^{\ast}_{\mathrm{INT}}(\mathsf{Q})$, then the cotangent lift of the left group action $g_j(s) =  1 + u_{1,j}(s) \xi^{-1} +u_{2,j}(s) \xi^{-2} $ on $\alpha_j$ is
\begin{align*}
    TL^{\ast}_{g_j}\alpha_j &=  \alpha_{2,j} \xi^2 +\left( u_{1,j}\alpha_{2,j} + \alpha_{1,j}  \right)\xi + u_{1,j}\alpha_{1,j} + u_{2,j}\alpha_{2,j} + 2\alpha_{2,j}\mathsf{D}_x u_{1,j}  \\
    &+ \left( u_{2,j}\alpha_{1,j}  + \alpha_{1,j}\mathsf{D}_x u_{1,j} + 1.0\alpha_{2,j}\mathsf{D}_x^2 u_{1,j}   + 2\alpha_{2,j}\mathsf{D}_x u_{2,j} \right)\xi^{-1} \\
    &+  \left( \alpha_{1,j}\mathsf{D}_x u_{2,j}  + 1.0\alpha_{2,j}\mathsf{D}_x^2 u_{2,j}  \right) \xi^{-2} +  \mathcal{O}(\xi^{-3}),
\end{align*}
in parallel, the cotangent lift of the right action is
\begin{align*}
    TR^{\ast}_{g_j }\alpha_j &= \alpha_{2,j} \xi^2 + \left( u_{1,j}\alpha_{2,j} + \alpha_{1,j}  \right) \xi+ u_{1,j}\alpha_{1,j} - u_{1,j}\mathsf{D}_x \alpha_{2,j} + u_{2,j}\alpha_{2,j}  \\
    &+ \left( - u_{1,j}\mathsf{D}_x \alpha_{1,j}  + 1.0u_{1,j}\mathsf{D}_x^2\alpha_{2,j}   + u_{2,j}\alpha_{1,j} - 2u_{2,j}\mathsf{D}_x \alpha_{2,j}  \right) \xi^{-1} \\
    &+  \left( u_{1,j}\mathsf{D}_x^2\alpha_{1,j}   - 2u_{2,j}\mathsf{D}_x \alpha_{1,j}   + 3.0u_{2,j}\mathsf{D}_x^2 \alpha_{2,j}  \right) \xi^{-2} + \mathcal{O}(\xi^{-3}).
\end{align*}
\end{definition}

In addition to the tangent and cotangent lifts, we need to define the discrete version of the representation of the action of $G_{\mathrm{INT}}(\mathsf{Q})$ on $\mathrm{Dens}(Q)$. One issue that arises is then in \eqref{eq:coadjoint_action_0}, the action is appears in Lagrangian coordinates, but the discrete setting is based on the Eulerian point of view. There are two ways we can resolve this issue: Using interpolation of the probability density function $\rho$ and numerically integrating the continuity equation in \eqref{eq:Lie_poisson_semi}. For the former, it is inspired by smoothed particle hydrodynamics \cite{monaghan1992smoothed}, where the elements of $\mathrm{Dens}(Q)$ are approximated by
\begin{align}
    \rho(x,t) \approx \sum\limits_j \rho_j(t) W_j(x).
\end{align}
Here, $W_j(x)$ are weights functions and one possible candidate of $W_j(x)$ is the Gaussian function
\begin{align*}
    W_j(x) = e^{- \alpha \| x - x_j \|^2}, \quad \text{with} \quad \alpha \in \mathbb{R}_{>0}.
\end{align*}
The second approach is to use a numerical scheme for approximating $\rho(x,t)$, where the \eqref{eq:Lie_poisson_semi} is solved in the interval $[t, t+ \Delta t]$.
\begin{definition}\label{def:pullback_rho_0}
Consider $g_j(s) \in G_{\mathrm{INT}}(\mathsf{Q})$ on $\mathrm{Dens}(\mathsf{Q})$, the function \eqref{eq:zeta}, the map $\zeta :G_{\mathrm{INT}}(\mathsf{Q})\to \mathrm{Dens}(\mathsf{Q})$ defined by
\begin{align} \label{eq:pullback_rho_0}
    \zeta(g_j) \rho = \rho_0(g_j \circ x)  = \sum\limits_j \rho_j(t) W_j\left(x + \int_0^t u_j \, dt\right).
\end{align}
On the other hand, it can be defined as one time iteration of the Fokker--Planck equation 
\begin{align}\label{eq:pullback_rho_1}
    \zeta(g_j) \rho =  \rho^{k} + \left( \Delta t \right) \mathsf{D}_x\left(\rho_{j}^ku_j \right).
\end{align}
\end{definition}
\begin{remark}
From a numerical stand point, using the pullback operator given in definition \eqref{eq:pullback_rho_0} is more stable and suitable to the continuity equation, however, the one given by definition \eqref{eq:pullback_rho_1} guarantees that we recover the continuity equation when taking $\rho^{k+1} - \rho^{k}$.
\end{remark}
The final term needed to explicitly apply the discrete Lie--Poisson Hamilton--Jacobi theory is the representation of the Lie group action on the dual of the Lie algebra.
\begin{definition}\label{def:pullback_rho_0_1}
Consider $g_j(s) \in \Psi G_d$ on $\mathrm{Dens}^{\ast}(\mathsf{Q})$, the discrete version of the map $(\zeta' )^{\ast} :\mathrm{Dens}^{\ast}(\mathsf{Q}) \to \mathfrak{g}_{\mathrm{INT}}(\mathsf{Q}) $ defined by
\begin{align} \label{eq:pullback_rho_0_1}
    (\zeta'_{f} )^{\ast} \rho = \rho^k_{j} \mathsf{D}_x f, 
\end{align}
for all $f \in \mathrm{Dens}^{\ast}(\mathsf{Q})$.
\end{definition}

\begin{proposition}\label{prop:numerical_LP}
The Hamiltonian system \eqref{eq:Lie_poisson_semi}, written in the Lie--Poisson Hamilton--Jacobi form \eqref{eq:Pi_0_mfg}, is solved numerically by
\begin{align}
    S_L &= S_0 + \Delta t H \left(  TL^{\ast}_{A}\mathbf{d}_{A}S_0 , \mathbf{d}_{\rho}S_0 \right) \\
    (M^{k}, \rho^{k}) &= \left(  \xi + u_j\xi^{-1} +  \mathsf{D}_x u_j \xi^{-2} + (\Delta t) \frac{\delta H}{\delta A} \xi + (\Delta t) u_j \frac{\delta H}{\delta A} \xi^{-1} \right. \nonumber \\
    &\left. + (\Delta t) \frac{\delta H}{\delta A} \mathsf{D}_x u_j \xi^{-2}   
 , \rho_{k} - \left( \Delta t \right) \mathsf{D}_x\left(\rho_{k,j}u_j\right) + (\Delta t)\frac{\delta H}{\delta \rho}  \right) \label{eq:rightMomentumMap_discrete} \\
    (M^{k+1}, \rho^{k+1}) &= \left(  \xi + u_j \xi^{-1} + (\Delta t) \frac{\delta H}{\delta A} \xi+ (\Delta t) u_j \frac{\delta H}{\delta A} \xi^{-1} - 2 (\Delta t) u_j \mathsf{D}_x \frac{\delta H}{\delta A} \xi^{-2} \right. \nonumber \\
    &\left.+ 3 (\Delta t) u_j \Delta^2_x \frac{\delta H}{\delta A} \xi^{-3}  + \rho_{k+1,j} \mathsf{D}_x f  , \rho + (\Delta t)\frac{\delta H}{\delta \rho} \right), \label{eq:leftMomentumMap_discrete}
\end{align}
with the initial condition $S_0 = \mathrm{Tr}(A) + \| \rho\|^2$.
\end{proposition}

One thing that needs to be pointed out,  that this formulation introduces a number of open problems regarding the inputs and target data. The use of probability density functions eliminates the role of inputs and targets and requires us to find suitable initial and final joint distributions in order to incorporate the training set in learning.  Nevertheless, one application which does not require picking suitable joint distributions is normalizing flow \cite{rezende2015variational}. In statistical inference, one requires to sample random data from various distributions and sometimes such distributions are complicated, thus renders standard techniques inefficient. In \cite{rezende2015variational}, it was proposed invertible maps, implemented using neural networks, to be applied successively as a way of mapping the target distribution to a simpler one, such as the normal Gaussian distribution. Once that is finished, sampling from the target distribution is done by simply sampling from the simple distribution and then map the sample points using the learned transformation to points that belong to the target distribution. The formulation of presented here lends itself to such application. Indeed, in \cite{G2021}, a stochastic ResNet whose parameters satisfy \eqref{eq:Lie_poisson_semi3} was used to learn four different non-Gaussian distributions and the results are shown in figure ~\ref{fig:normalizing_flow}. Although, the numerical scheme used in \cite{G2021} uses a multisymplectic variational integrator.

\begin{figure}[h]
\centering
    \includegraphics[width=0.8\textwidth]{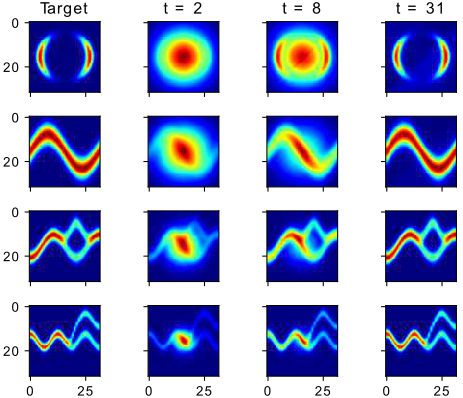}
    \caption{This numerical experiment is reproduction of the result in \cite{rezende2015variational}. A stochastic ResNet is used for a normalizing flow example. The idea is to find network parameters which would map a normal Gaussian distribution to a target distribution, that is not necessarily Gaussian.  Each row is a separate distribution and the left most column is the target distribution. The other columns are snapshots of the progression a normal distribution through the layers of the stochastic ResNet. A mean--field control is used to solve the stochastic optimal control problem used to train the network and the parameters evolve according to \eqref{eq:Lie_poisson_semi3}.  A multisymplectic numerical scheme for deep learning that solves \eqref{eq:Lie_poisson_semi3}, presented in \cite{G2021}, was used to obtain this reproduction. }
    \label{fig:normalizing_flow}
\end{figure}

\section{Connection to Integrable Partial Differential Equations}
In many cases, substituting the optimality condition $\partial_u H = 0$ into the Pontryagin's Hamiltonian yields,  an integrable system \cite{jurdjevic_integrable_2011, jurdjevic_2016}. This line of investigation is interesting for deep learning, as it can be beneficial for studying various behaviours. Apart from knowing the solution of the integrable systems, they do posses interesting properties, such as the bi-Hamiltonian structure. The existence of such structure means that the problem has another Hamiltonian function and according to Magri's lemma \cite{magri1978simple}, we can construct an infinite number of conservation laws by simply knowing the two Poisson structures. 
\\

Here, we use Lie--Poisson formulation to show that deep learning training, with certain conditions on the neural network's parameters, is equivalent to solving the nonlinear Schr\"{o}dinger equation. One very interesting fact regarding the Hamilton--Jacobi equation is that it works as a bridge from classical to quantum mechanics \cite{goldstein1980classical}. This is done by writing the generating function in terms of the wave function, $\psi$, i.e. 
\begin{align}\label{eq:HJ_to_Seqn}
\lambda (x,t) = \frac{1}{i}\mathrm{log} \psi(x,t),     
\end{align}
and substituting it into the Hamilton--Jacobi equation and the result is the Schr\"{o}dinger equation. For the system of equations \eqref{eq:opt_cond_mfg}, \eqref{eq:fokker_planck} and \eqref{eq:stoch_hjb}, instead of \eqref{eq:HJ_to_Seqn}, the following wave function is used
\begin{align}\label{eq:madelung_trans}
    \psi = \sqrt{\rho(x,t)} e^{i \lambda(x,t)}
\end{align}
where $\rho$ satisfy \eqref{eq:fokker_planck} and $\lambda$ satisfy \eqref{eq:stoch_hjb}, and this transformation is called \emph{Madelung transformation}  and it casts the nonlinear Schr\"{o}dinger equation as the \emph{quantum Euler equation} \cite{madelung_quantentheorie_1927}. This transformation has gained attention in recent years \cite{fusca_madelung_2017, khesin_geometric_2018, khesin_geometric_2020} as it defines a symplectic transformation from the Hilbert space of complex-valued wave functions $H^{\infty}(\mathbb{R}^n, \mathbb{C})$ to the dual of semidirect product Lie algebra $\mathfrak{X}^{\ast}(\mathbb{R}^n) \times H^{\infty}(\mathbb{R}^n)$.
\\

\begin{theorem}
Consider the family of Schr\"{o}diner equations 
\begin{align}\label{eq:schrodinger_eqn}
i \hbar \partial_t \psi + \frac{\hbar^2}{2m} \Delta \psi  + f(|\psi|^2) \psi,
\end{align}
where $\hbar$ is the Planck's constant, $m$ is the mass and $f: \mathbb{R}_{>0} \to \mathbb{R}$ and whose Hamiltonian on the Hilbert space for complex-valued wave functions $H_{NLS}: H^{\infty}(Q, \mathbb{C}) \to \mathbb{R}$ is given by
\begin{align} \label{eq:ham_schrodinger}
	H(\psi) =   \frac{\hbar}{2m} \int_{Q} \| \nabla \psi \|^2 \, dx+ \int_{Q} F(| \psi |^2) \, dx. 
\end{align}
With $f = 0$, it is equivalent to the training of deep residual neural networks 
\begin{align}
J(\theta) =\int_0^T \ell_d \left(q(t), \theta(t)\right) \, dt + \sum_{i=0}^N L( \pi((q(T)^{(i)}), c^{(i)} ),
\end{align}
constrained to 
\begin{align}\label{eq:sde_control_1}
    d_tq^{(i)}(t) = \theta(t) \circ q^{(i)}(t)  \, dt + \nu dW_t, \quad q^{(i)}(0) = q^{(i)}_0, \, \pi(q^{(i)}(T)) = c^{(i)},
\end{align}
with $i= 1, \dots, N$, where $( q^{(i)}_0, c^{(i)})$ form the dataset used in supervised training. The main assumption imposed here is that 
\begin{align}\label{eq:incompressible_condition}
    \mathrm{div}(\theta)  - \frac{\nu^2}{2}\mathrm{div}\left(\nabla \mathrm{log}(\rho)\right) = 0
\end{align}
\end{theorem}
\begin{proof}
The direct equivalence between the two problems is not direct, so the first step of the proof is to cast the training problem as a mean--field games problem. Instead of controlling \eqref{eq:sde_control_1}, the probability density function is controlled. The dynamical constraint for the probability density function is the Fokker--Planck equation
\begin{align}\label{eq:fokker_planck2}
    \partial_t \rho + \mathrm{div}\left( \rho \sigma(x, \theta)\right) - \frac{\nu^2}{2} \Delta \rho = 0,
\end{align}
with $\rho(x,0) = \rho_0$ and $\rho(x,T) = \rho_T$. Before solving the optimal control problem for the probability density, as mentioned in section \ref{sec:deep_poisson}, a change of variable is performed. The network parameters are decomposed into 
\begin{align}
    \theta(x, t)  = \omega(x, t) + \frac{\nu^2}{2} \nabla \mathrm{log}{\rho},
    \end{align}
and substituting it into \eqref{eq:fokker_planck2}, we obtain the continuity equation
\begin{align}
    \partial_t \rho + \mathrm{div} \left( \rho \omega(x, t) \right) = 0.
\end{align}
The Pontryagin's Hamiltonian then becomes
\begin{align*}
    H(\rho, \lambda) = \frac{1}{2} \int_{Q} \rho \| \omega \|^2 d\mathrm{Vol}(x) +\frac{\nu^2}{2}  \int_{Q}  \rho \left\langle \omega, \nabla \mathrm{log}(\rho) \right\rangle  d\mathrm{Vol}(x) + \frac{\nu^4}{8} \int_{Q} \rho \|  \nabla \mathrm{log}(\rho)\|^2 d\mathrm{Vol}(x)
\end{align*}
From the assumption \eqref{eq:incompressible_condition} we have the condition that $\mathrm{div} \omega = 0$  as a result, the Hamiltonian further simplifies to 
\begin{align} \label{eq:Ham_mfg_3}
    H(\rho, \lambda) = \frac{1}{2} \int_{Q} \rho \| \omega \|^2 d\mathrm{Vol}(x) + \frac{\nu^4}{8} \int_{Q} \rho \|  \nabla \mathrm{log}(\rho)\|^2 d\mathrm{Vol}(x).
\end{align}
The Hamiltonian function \eqref{eq:Ham_mfg_3} is obtained by substituting the Madelung transformation \eqref{eq:madelung_trans}, which is a mapping from $H^{\infty}(Q, \mathbb{C})$ to $\mathfrak{X}^{\ast}(Q) \times \mathrm{Dens}(Q)$, into the Hamiltonian, $H_{NLS}$, for the Schr\"{o}dinger equation \eqref{eq:schrodinger_eqn}, we have
\begin{equation}
\begin{aligned}
    H(\rho, \lambda)& =  \int_{Q} \frac{\hbar}{2 m } \left\langle \frac{\nabla \rho}{2 \sqrt{\rho}} e^{i\frac{\lambda}{\sqrt{\hbar}}} + \frac{i}{\sqrt{\hbar}}\nabla \lambda \sqrt{\rho} e^{i\frac{\lambda}{\sqrt{\hbar}}},\frac{\nabla \rho}{2 \sqrt{\rho}} e^{-i\frac{\lambda}{\sqrt{\hbar}}} - \frac{i}{\sqrt{\hbar}} \nabla \lambda \sqrt{\rho} e^{-i\frac{\lambda}{\sqrt{\hbar}}} \right\rangle \, dx \\
    &=  \int_{Q} \frac{\hbar}{8 } \frac{ \left\|\nabla \rho\right\|^2}{{\rho}} + \frac{1}{2}\rho \left\| \nabla \lambda \right\|^2 \, dx
\end{aligned}
\end{equation}
if $\hbar = \nu^4$ and using \eqref{eq:opt_cond_mfg}, i.e. $\nabla \lambda = \omega$, we obtain the Hamiltonian \eqref{eq:Ham_mfg_3}.
\end{proof}
\begin{corollary}
Consider the deep neural network training problem
\begin{align}
J(\theta) =\int_0^T \ell_d \left(q(t), \theta(t)\right) \, dt + \sum_{i=0}^N L( \pi((q(T)^{(i)}), c^{(i)} ),
\end{align} 
constrained to the ensemble of stochastic differential equations 
\begin{align}\label{eq:sde_control}
    d_tq^{(i)}(t) = \theta(t) \circ q^{(i)}(t)  \, dt + \nu dW_t, \quad q^{(i)}(0) = q^{(i)}_0, \, \pi(q^{(i)}(T)) = c^{(i)},
\end{align}
and the pair of initial and final conditions are the training dataset $(q^{(i)}_0, c^{(i)})$, where $i= 1, \dots, N$. The stochastic control problem is equivalent to the Lie--Poisson system with Hamiltonian $H: \Psi \mathrm{DS}(\mathbb{C})^{\ast} \to \mathbb{R}$ given by
\begin{align}
    H(L) = \frac{i}{3} \mathrm{Tr}(L^3),
\end{align}
where $L \in \Psi \mathrm{DS}(\mathbb{C})^{\ast}$ is defined as
\begin{align}
    L = \partial_x + \psi \partial_x^{-1} \psi^{\ast}.
\end{align}
\end{corollary}
With the Hamiltonian defined on $\Psi \mathrm{DS}(\mathbb{C})^{\ast}$, the dual algebra of complex-valued pseudodifferential symbols, we can simply apply theorem \ref{prop:SemiLPHJ}.
\\

The he nonlinear Schr\"{o}dinger equation is not only integrable equation that arises from deep learning, in fact, when equation \eqref{eq:Lie_poisson_semi3}, when $\nu \to 0$, and the network's parameters regularising norm $\ell$ is chosen as $H^1$ Sobolev norm, the equation becomes
\begin{align*}
\partial_t m &=  \mathrm{ad}^{\ast}_{u} m, \\
m &= u - \alpha^2 \partial_x^2 u
\end{align*}
or when expressed in coordinates on $\mathbb{R}^N$, it becomes
\begin{equation}
\begin{aligned}
\partial_t m_i &+  \partial_{x^j}\left( m_i u^j \right) + \left(m_j \right)  \partial_{x^i}u^j  = 0 \\
m_i &= u_i - \alpha^2 \partial^2_{x^i}u_i,
\end{aligned}
\end{equation}
and it is the \emph{dispersionless Camassa--Holm equation}  \cite{camassa_integrable_1993} and it is an example of a Lie--Poisson system. An interesting property of this partial differential equation is that its weak solution is also a Hamiltonian system 
\begin{align}\label{eq:particle_solution}
H(q, p, t) &=  \sum\limits_{i=0, k=0}^{N_b, N_b} p^{(i)}p^{(k)}e^{- \frac{1}{2} \| q^{(i)} - q^{(k)} \|^2} 
\end{align}
The physics that are described by this Hamiltonian is an approximation to the N-particle solution, or \emph{peakon train}, of Euler--Poincar\'{e} equation with $H^1$ Sobolev norm \cite{kouranbaeva_camassaholm_1999, misio_lek_shallow_1998}.  What is meant by the peakon or peaked soliton is the a type of soliton, but with a discontinuity or a cusp at the crest of the wave. For the actual peakon train, the Gaussian in the Hamiltonian \eqref{eq:particle_solution}, is replaced by Green's function  $G(x,y) = e^{-|x - y|}x$ for shifted Poisson differential operator $I_d - \partial^2_{x}$.
\\

Further, the Hamilton's equation for the peakon train do have a Lax pair, and it is an example of an integrable system.
\begin{align}
    \frac{d}{dt}L = [L, P], 
\end{align}
where $L$ an $N\times N$ matrix, known as Lax matrix, with entries $L_{i,j} = \sqrt{p^{(i)}p^{(k)}}e^{-|q^{(i)} - q^{(k)}|}$ and $P$ is also an $N \times N$ matrix and its entries are $P_{i,j} = -2\sqrt{p^{(i)}p^{(k)}}\mathrm{sign}(q^{(i)} - q^{(k)})e^{-|q^{(i)} - q^{(k)}|}$. One advantage of knowing the Lax matrix is that taking trace of its power, i.e. $\mathrm{Tr}(L^k)$ for $k=2, \dots, N$, results in the constant of motion \cite{camassa_integrable_1993}. 
\\
\section{Conclusion \& Future Direction}
We have introduced, in this paper, an equivalent formulation of deep learning as a hydrodynamics system. This made possible by substituting the deep neural network by an ordinary differential equation and formulating the training of deep learning as an optimal control problem as it is equivalent to solving a Hamiltonian system.  First, it was the use of mean--field type control for the stochastic optimal control problem and that approach allows for the application of Pontryagin's Maximum Principle. The Hamilton's equations that are the optimality condition for deep learning, in this case, are partial differential equations, and instead of using them directly, we have used an equivalent system of equations that is similar to quantum fluid equation. Here we have demonstrated that this optimality condition is a Lie--Poisson system. Interestingly, the Poisson structure for deep learning coincides with the Poisson structure for the ideal compressible fluid. 
\\

What is interesting about the Lie--Poisson formulation is that it links deep learning with equations from fluid dynamics, in some specific cases, to the nonlinear Schr\"{o}dinger equation. Such connection have not been well studied, and it is a future research direction as the nonlinear Schr\"{o}dinger equation, in one-dimension at least, has a hierarchy of integrable equations. Potentially, this could help deriving conserved quantities for deep learning. The importance of conserved quantities is they allow us to better analyse the performance of the system and construct Lyapunov functions for studying the stability of the system; this approach forms the basis of Arnold's method for stability.   
\\

Despite the fact Lie--Poisson systems are Hamiltonian and with a Poisson structure, it is challenging to solve numerically and for this structure--preserving integrators are required. One type of structure--preserving integrators we have explored in this paper is integrators based on the Hamilton--Jacobi theory and its symmetry reduced version the Lie--Poisson Hamilton--Jacobi theory. The main reason for this choice is it can be used for both Hamiltonian systems on symplectic cotangent bundles and Poisson manifolds. Moreover, one of the advantages of this approach is that can be easily be made high-order without the need to impose additional constraints to enforce the conservation of the symplectic structure. This forms the basis of a training algorithm, and we have shown the algorithm behaves such similar methods using Euler method and $4^{th}$ order Runge--Kutta method.  
\\

The application is not straight forward and as a result there has been a number of modifications are needed. On top of that, the underlying dynamics are defined on infinite dimensional semidirect product space. For that reason, we have extended the Lie--Poisson Hamilton--Jacobi theory to accommodate semidirect product space and we have also formulated the hydrodynamics problem on the group of pseudodifferential symbols. The last step is required as it allows us to apply Lie--Poisson Hamilton--Jacobi theory directly to the quantum fluid equation for deep learning and derive structure--preserving numerical schemes. 
\\

\textbf{Stability of Deep Learning:} As mentioned, the Arnold's method of stability can be applied to the nonlinear Schr\"{o}dinger equation as a surrogate way to study the stability of deep learning, however, other methods that reply on the geometry of the problem can be used. One of such methods is the energy--Casimir method of \cite{HOLM19851, arnaudon2018stochastic}, where it is used to study the nonlinear stability of Lie--Poisson systems' relative equilibria and it is applied to various ideal compressible fluid and even plasma physics. The notion of nonlinear stability for a trajectory $(M(t), \rho(t)) \in \mathfrak{s}^{\ast}$ is that there exists $\epsilon > 0$ such that $\| (M_0, \rho_0) \| < \epsilon$, then $ \| ((M(t) - M_0, \rho(t) - \rho_0) \| < \delta$ for all of $t$, where $\delta > 0$. In other words, the trajectory, whose initial condition belongs to a specific set, would remain close to the set of initial values. The main ingredient is the Casimir function, which is a map $C: \mathfrak{s}^{\ast} \to \mathbb{R}$ such that it is the centre of the Lie algebra meaning for every function $F: \mathfrak{s}^{\ast} \to \mathbb{R}$, the Lie--Poisson bracket is equal to zero, i.e. $\{C, F \} = 0$. The Hamiltonian is then augmented with the Casimir function to form a new Hamiltonian $H_C = H + \Phi(C)$ and first variation of the Hamiltonian is evaluated at the equilibrium and then equated to zero. From there we find the function $\Phi$. The next step is to determine the convexity of the second variation of the Hamiltonian $H_C$. Once all criteria are met, one can determine the nonlinear stability of the system.
\\

The major challenge with the energy--Casimir method is that it requires a knowledge of the Casimir method, and this motivated the development of energy--momentum method of \cite{Simo1991StabilityOR}. Instead, the augmented Hamiltonian is $H_{\xi} = H + \left\langle \mathbf{J},  \xi \right\rangle$, where $\mathbf{J}: T^{\ast} Q \to \mathfrak{s}^{\ast}$ and $\xi \in \mathfrak{s}$. The Lie algebra element $\xi$ plays the role of the vector field associated with the relative equilibria and it renders the first variation $\delta H_{\xi} = 0$. Same as energy--Casimir method, the definiteness of $\delta^2 H_{\xi}$ indicates the stability of the system. 
\\

Another route for studying stability of Hamiltonian systems is to find the action-angle variables, denoted by $(I_1, \dots, I_n, \theta_1, \dots, \theta_n)$ In this coordinate system, the Hamiltonian $H$ depends only on the actions $I_1, \dots, I_n$ and the result of that is the configuration manifold is an n-torus $\mathbb{T}^n$. These variables are needed in order to study the perturbation of the Hamiltonian and we can assess the stability by KAM theory \cite{kolmogorov1954conservation, moser_invariant_1962, Arnold1963PROOFOA}. That said, KAM theory works on integrable systems and, at this point it is yet to be investigates, it is not known whether or not the deep learning system is integrable or not. What we know is that for a specific case, we can the Hamiltonian obtained from Pontryagin's Maximim Principle is integrable, which leads us to the next possible research direction. 
\\

\textbf{Generating Function For Optimal Control:} The issue with using structure--preserving integrators and gradient methods is that they require an initial guess for the costate and control parameters. Such guesses can affect the behaviour of the Hamiltonian system and convergence in general and in addition, iterative schemes for such problems can be time consuming and computationally taxing.  A different approach that is related to dynamic programming was proposed in \cite{guibout_formation_2002} and further investigated in \cite{park2003solutions,Guibout04thehamilton-jacobi, scheeres2003solving, guibout_solving_2004}. The key ingredient is to compute the generating function for symplectomorphism that maps between the initial and final conditions. This further makes it possible to find the solution of the costate in terms of training dataset, thus sidestepping the need to guess the initial/final condition for the costate. Such research direction is worth investigating for training of deep neural networks. 
\\

The idea is to treat initial and terminal conditions, $q(0) = q_0 \in Q$ and $q(T) = q_T \in Q$ respectively, as old and new coordinates, $(q_0, p_0)$ and $(q_T, p_T)$.  To see how generating functions and optimal transport are related, we start by looking at dynamic programming approach to optimal control. Given the cost functional 
\begin{align}
\mathbb{S}(q) = \int_0^T \ell(\theta) \, dt + \Phi(q(T), q_T),
\end{align}
we define the \emph{value function} as 
\begin{align}
V(q(t), t) = \min\limits_{\theta} \left\{ V(q(t + dt), t+dt) +  \int_t^{t+ dt} \ell(\theta) \, dt \right\},
\end{align}
and it satisfies the \emph{Hamilton--Jacobi--Bellman equation}, which can be regarded as a generalisation of the Hamilton--Jacobi equation. As shown in (Park), the value function is
\begin{align}\label{eq:value_func}
V(q(t), t) = -S(q(t), q_T,t) + \Phi(q(T), q_T)
\end{align}
which mean that it can be obtained by solving the Hamilton--Jacobi equation for $S$. That said, one small technical issue: it is not practical to use $S$ directly and that is because at $t=0$ it is hard to find initial conditions that satisfy the identity transformation. For this reason, we use type--II generating function $S_2$, or type--III generating function $S_3$ instead. Once $V$ is known, the parameters that would train the network are then given by
\begin{align}\label{eq:optimality_value_func}
\theta =  \arg\max\limits_{\theta^{\ast}} \left( H\left(q, \frac{\partial V}{\partial q}, \theta^{\ast}, t\right) \right)
\end{align}

That said, finding a solution for the generating function is challenging on its own right, however, in \cite{guibout_formation_2002} it was shown that it can be solved using power series expansion of generating function whose coefficients are obtained by solving a system of ordinary differential equations. The derivation of the system of ordinary differential equations is done by applying the Taylor expansion to the Hamiltonian and the generating function. Once that is done, both expansions are substituted into the Hamilton--Jacobi equation and as a result we obtain a power series whose coefficients are differential equations. Setting the coefficients of the combined series to be equal to zero, we obtain the system of equations whose solution is the coefficients of the generating function $S$. 
\\

The $N$ terms Taylor expansion of the Hamiltonian, around the null point of $T^{\ast}Q$, is given by
\begin{align*}
\widetilde{H}(q,p, t) = \sum_{i_{q_1}=0}^N \cdots \sum_{i_{q_1}=0}^N  \frac{(q_1)^{i_{q_1}}\cdots (q_n)^{i_{q_n}}(p_1)^{i_{p_1}}\cdots (p_n)^{i_{p_n}}}{{i_{q_1}}!\cdots i_{p_1}!}\,\left(\frac{\partial^{i_{q_1} + \cdots + i_{p_n}}}{\partial q_1^{i_{q_1}}\cdots \partial p_n^{i_{p_n}}}H\right)(0,\ldots,0,t),
\end{align*}
likewise, the generating function, is written as a polynomial with $N$ terms,
\begin{align*}
F(q,P, t) = \sum_{i_{q_1}=0}^N \cdots \sum_{i_{q_1}=0}^N  \frac{(q_1)^{i_{q_1}}\cdots (q_n)^{i_{q_n}}(p_1)^{i_{p_1}}\cdots (p_n)^{i_{p_n}}}{{i_{q_1}}!\cdots i_{p_1}!}\,F_{i_{q_1}, \ldots,i_{p_n} }(t),
\end{align*}
where $F_{i_{q_1}, \ldots,i_{p_n} }(t)\in \mathcal{C}(Q)$ are the coefficients of polynomial $F$ and they are differentiable functions on $Q$. Substituting the expansions $F(q,P,t)$ and $\widetilde{H}(q,p, t)$ into the Hamilton--Jacobi equation, we obtain a system of differential equations for the coefficients of $F(q,P, t)$. Once the solution is obtained, that yields the generating function $F(q,P, t) $, which is then used to obtain \eqref{eq:value_func}, which in turn allows us to obtain the weights of the network \eqref{eq:optimality_value_func}.
\\

This formulation is in the continuous-time space and for large systems, it is practical to use discrete time formulation. Mirroring the approach in the previous section, a discrete version of the generating function approach for solving two-boundary point problems could be beneficial. The Hamilton--Jacobi equation becomes a discrete one and it was introduced in \cite{Ohsawa2011DiscreteHT}. Given the set of discrete points of the generating function $\{ S^k \}_{k=1}^N$ which satisfy the discrete Hamilton--Jacobi equation 
\begin{align} \label{eq:discrete_HJ_eqn}
S^{k+1}(q_{k+1}) = S^{k}(q_{k+1})+ D_2S^{k+1}(q_{k+1}) \cdot q_{k+1} + H^+_d(q_{k}, D_2S^{k+1}(q_{k+1})), 
\end{align}
with the symplectic mapping $(q_k, p_k) \to (q_{k+1}, p_{k+1})$ defined by
\begin{align}
q_{k+1} = q_k + D_1S^{k+1}(q_{k+1}), \\
p_{k+1} = p_k - D_2S^{k+1}(q_{k+1}).
\end{align}
Once $\{ S^k \}_{k=1}^N$ are obtained, we can then use \eqref{eq:value_func} to obtain $V$ and from it $\theta$ of \eqref{eq:optimality_value_func}.
\bibliography{references}

\begin{thebibliography}{100}

\bibitem{abia_partitioned_1993}
L.~Abia and J.~M. Sanz-Serna.
\newblock Partitioned {Runge}-{Kutta} {Methods} for {Separable} {Hamiltonian}
  {Problems}.
\newblock {\em Mathematics of Computation}, 60(202):617--634, 1993.
\newblock Publisher: American Mathematical Society.

\bibitem{abraham_manifolds_1988}
R.~Abraham, J.~E. Marsden, and T.~Ratiu.
\newblock {\em Manifolds, {Tensor} {Analysis}, and {Applications}}.
\newblock Applied {Mathematical} {Sciences}. Springer-Verlag, New York, 2
  edition, 1988.

\bibitem{agrachev2004control}
A.~Agrachev, M.~Zelikin, Y.~Sachkov, Y.~Sachkov, and {\^U}.~Sa{\v{c}}kov.
\newblock {\em Control Theory from the Geometric Viewpoint}.
\newblock Control theory and optimization. Springer, 2004.

\bibitem{arnaudon_stochastic_2018}
A.~Arnaudon, N.~Ganaba, and D.~D. Holm.
\newblock The stochastic energy-{Casimir} method.
\newblock {\em Comptes Rendus M{\'e}canique}, 346(4):279--290, Apr. 2018.

\bibitem{arnaudon2018stochastic}
A.~Arnaudon, N.~Ganaba, and D.~D. Holm.
\newblock The stochastic energy-casimir method.
\newblock {\em Comptes Rendus M{\'e}canique}, 346(4):279--290, 2018.

\bibitem{arnold1966geometrie}
V.~Arnold.
\newblock Sur la g{\'e}om{\'e}trie diff{\'e}rentielle des groupes de lie de
  dimension infinie et ses applications {\`a} l'hydrodynamique des fluides
  parfaits.
\newblock {\em Annales de l'Institut Fourier}, 16(1):319--361, 1966.

\bibitem{Arnold1963PROOFOA}
V.~I. Arnol'd.
\newblock Proof of a theorem of {A}. {N}. {K}olmogorov on the invariance of
  quasi-periodic motions under small perturbations of the {H}amiltonian.
\newblock {\em Russian Mathematical Surveys}, 18:9--36, 1963.

\bibitem{barash_deciphering_2010}
Y.~Barash, J.~A. Calarco, W.~Gao, Q.~Pan, X.~Wang, O.~Shai, B.~J. Blencowe, and
  B.~J. Frey.
\newblock Deciphering the splicing code.
\newblock {\em Nature}, 465(7294):53--59, May 2010.
\newblock Number: 7294 Publisher: Nature Publishing Group.

\bibitem{barbaresco2020lie}
F.~Barbaresco and F.~Gay-Balmaz.
\newblock Lie group cohomology and (multi)symplectic integrators: New geometric
  tools for lie group machine learning based on souriau geometric statistical
  mechanics.
\newblock {\em Entropy}, 22(5), 2020.

\bibitem{benamou2000computational}
J.-D. Benamou and Y.~Brenier.
\newblock A computational fluid mechanics solution to the {Monge-Kantorovich}
  mass transfer problem.
\newblock {\em Numerische Mathematik}, 84(3):375--393, 2000.

\bibitem{benning2019deep}
M.~Benning, E.~Celledoni, M.~J. Ehrhardt, B.~Owren, and C.-B. Sch\"{o}nlieb.
\newblock Deep learning as optimal control problems: Models and numerical
  methods.
\newblock {\em Journal of Computational Dynamics}, 6(2):171--198, 2019.

\bibitem{bensoussan_mean_2013}
A.~Bensoussan, J.~Frehse, and P.~Yam.
\newblock The {Mean} {Field} {Type} {Control} {Problems}.
\newblock In A.~Bensoussan, J.~Frehse, and P.~Yam, editors, {\em Mean {Field}
  {Games} and {Mean} {Field} {Type} {Control} {Theory}}, {SpringerBriefs} in
  {Mathematics}, pages 15--29. Springer, New York, NY, 2013.

\bibitem{blaszak2012multi}
M.~Blaszak.
\newblock {\em Multi-Hamiltonian Theory of Dynamical Systems}.
\newblock Theoretical and Mathematical Physics. Springer Berlin Heidelberg,
  2012.

\bibitem{bloch_optimal_1999}
A.~M. Bloch and P.~E. Crouch.
\newblock Optimal {Control}, {Optimization}, and {Analytical} {Mechanics}.
\newblock In J.~Baillieul and J.~C. Willems, editors, {\em Mathematical
  {Control} {Theory}}, pages 268--321. Springer, New York, NY, 1999.

\bibitem{bloch2000optimal}
A.~M. Bloch, D.~D. Holm, P.~E. Crouch, and J.~E. Marsden.
\newblock An optimal control formulation for inviscid incompressible ideal
  fluid flow.
\newblock In {\em Proceedings of the 39th IEEE Conference on Decision and
  Control (Cat. No. 00CH37187)}, volume~2, pages 1273--1278. IEEE, 2000.

\bibitem{bloch_geometric_2009}
A.~M. Bloch, I.~I. Hussein, M.~Leok, and A.~K. Sanyal.
\newblock Geometric structure-preserving optimal control of a rigid body.
\newblock {\em Journal of Dynamical and Control Systems}, 15(3):307--330, July
  2009.

\bibitem{1980719}
M.~Born and E.~Wolf.
\newblock Principles of optics: Electromagnetic theory of propagation,
  interference and diffraction of light.
\newblock Elsevier Science, 2013.

\bibitem{FB/ADL:04supp}
F.~Bullo and A.~D. Lewis.
\newblock Supplementary chapters for \emph{{G}eometric {C}ontrol of
  {M}echanical {S}ystems}~\cite{FB/ADL:04}, jan 2005.

\bibitem{camassa_integrable_1993}
R.~Camassa and D.~D. Holm.
\newblock An integrable shallow water equation with peaked solitons.
\newblock {\em Physical Review Letters}, 71(11):1661--1664, Sept. 1993.
\newblock Publisher: American Physical Society.

\bibitem{caratheodory_grundlagen_1937}
C.~Carath\'{e}odory.
\newblock Die {Grundlagen} der geometrischen {Optik}.
\newblock In C.~Carath\'{e}odory, editor, {\em Geometrische {Optik}},
  Ergebnisse der {Mathematik} und {Ihrer} {Grenzgebiete}, pages 15--35.
  Springer, Berlin, Heidelberg, 1937.

\bibitem{celledoni_structure_2020}
E.~Celledoni, M.~J. Ehrhardt, C.~Etmann, R.~I. McLachlan, B.~Owren, C.-B.
  Sch{\"o}nlieb, and F.~Sherry.
\newblock Structure preserving deep learning.
\newblock {\em arXiv:2006.03364 [cs, math, stat]}, June 2020.
\newblock arXiv: 2006.03364.

\bibitem{channell_symplectic_1990}
P.~J. Channell and C.~Scovel.
\newblock Symplectic integration of {Hamiltonian} systems.
\newblock {\em Nonlinearity}, 3(2):231--259, May 1990.

\bibitem{chen_rise_2018}
H.~Chen, O.~Engkvist, Y.~Wang, M.~Olivecrona, and T.~Blaschke.
\newblock The rise of deep learning in drug discovery.
\newblock {\em Drug Discovery Today}, 23(6):1241--1250, June 2018.

\bibitem{chen2006two}
M.~{Chen}, S.-Q. {Liu}, and Y.~{Zhang}.
\newblock A two-component generalization of the camassa-holm equation and its
  solutions.
\newblock {\em Letters in Mathematical Physics}, 75(1):1--15, Jan. 2006.

\bibitem{NIPS2018_7892}
T.~Q. Chen, Y.~Rubanova, J.~Bettencourt, and D.~K. Duvenaud.
\newblock Neural ordinary differential equations.
\newblock In S.~Bengio, H.~Wallach, H.~Larochelle, K.~Grauman, N.~Cesa-Bianchi,
  and R.~Garnett, editors, {\em Advances in Neural Information Processing
  Systems 31}, pages 6571--6583. Curran Associates, Inc., 2018.

\bibitem{doi:10.1098/rspa.2007.1892}
C.~Cotter, D.~Holm, and P.~Hydon.
\newblock Multisymplectic formulation of fluid dynamics using the inverse map.
\newblock {\em Proceedings of the Royal Society A: Mathematical, Physical and
  Engineering Sciences}, 463(2086):2671--2687, 2007.

\bibitem{dupont2019augmented}
E.~Dupont, A.~Doucet, and Y.~W. Teh.
\newblock Augmented neural odes.
\newblock {\em arXiv preprint arXiv:1904.01681}, 2019.

\bibitem{e_proposal_2017}
W.~E.
\newblock A {Proposal} on {Machine} {Learning} via {Dynamical} {Systems}.
\newblock {\em Communications in Mathematics and Statistics}, 5(1):1--11, mar
  2017.

\bibitem{e_mean-field_2018}
W.~E, J.~Han, and Q.~Li.
\newblock A mean-field optimal control formulation of deep learning.
\newblock {\em Research in the Mathematical Sciences}, 6(1):10, dec 2018.

\bibitem{10.2307/1970699}
D.~G. Ebin and J.~Marsden.
\newblock Groups of diffeomorphisms and the motion of an incompressible fluid.
\newblock {\em Annals of Mathematics}, 92(1):102--163, 1970.

\bibitem{engo_numerical_2001}
K.~Eng\o{} and S.~Faltinsen.
\newblock Numerical {Integration} of {Lie}--{Poisson} {Systems} {While}
  {Preserving} {Coadjoint} {Orbits} and {Energy}.
\newblock {\em SIAM Journal on Numerical Analysis}, 39(1):128--145, Jan. 2001.
\newblock Publisher: Society for Industrial and Applied Mathematics.

\bibitem{Falqui_2005}
G.~Falqui.
\newblock On a camassa{\textendash}holm type equation with two dependent
  variables.
\newblock {\em Journal of Physics A: Mathematical and General}, 39(2):327--342,
  dec 2005.

\bibitem{farabet_learning_2013}
C.~Farabet, C.~Couprie, L.~Najman, and Y.~LeCun.
\newblock Learning {Hierarchical} {Features} for {Scene} {Labeling}.
\newblock {\em IEEE Transactions on Pattern Analysis and Machine Intelligence},
  35(8):1915--1929, Aug. 2013.
\newblock Conference Name: IEEE Transactions on Pattern Analysis and Machine
  Intelligence.

\bibitem{fusca_madelung_2017}
D.~Fusca.
\newblock The {Madelung} transform as a momentum map.
\newblock {\em Journal of Geometric Mechanics}, 9(2):157, 2017.

\bibitem{G2021}
N.~Ganaba.
\newblock Multisymplectic formulation of deep learning.
\newblock In preparation, 2021.

\bibitem{G2020}
N.~Ganaba.
\newblock {S}tochastic {L}ie-{P}oisson-{H}amilton-{J}acobi {T}heory.
\newblock In preparation, 2021.

\bibitem{gawehn_deep_2016}
E.~Gawehn, J.~A. Hiss, and G.~Schneider.
\newblock Deep {Learning} in {Drug} {Discovery}.
\newblock {\em Molecular Informatics}, 35(1):3--14, 2016.
\newblock \_eprint:
  https://onlinelibrary.wiley.com/doi/pdf/10.1002/minf.201501008.

\bibitem{gay-balmaz_clebsch_2011}
F.~Gay-Balmaz and T.~S. Ratiu.
\newblock Clebsch optimal control formulation in mechanics.
\newblock {\em Journal of Geometric Mechanics}, 3(1):41, 2011.

\bibitem{goldstein1980classical}
H.~Goldstein and C.~Poole.
\newblock {\em Classical Mechanics}.
\newblock Addison-Wesley series in physics. Addison-Wesley Publishing Company,
  1980.

\bibitem{guibout_formation_2002}
V.~Guibout and D.~Scheeres.
\newblock Formation {Flight} with {Generating} {Functions}: {Solving} the
  {Relative} {Boundary} {Value} {Problem}.
\newblock In {\em {AIAA}/{AAS} {Astrodynamics} {Specialist} {Conference} and
  {Exhibit}}, Monterey, California, Aug. 2002. American Institute of
  Aeronautics and Astronautics.

\bibitem{Guibout04thehamilton-jacobi}
V.~M. Guibout.
\newblock The hamilton-jacobi theory for solving two-point-boundary
  value-problems: Theory and numerics with application to spacecraft formation
  flight.
\newblock In {\em Optimal Control, and the Study of Phase Space Structure.
  University of Michigan at Ann Arbor, Ph. D. Dissertation}, 2004.

\bibitem{guibout_solving_2004}
V.~M. Guibout and D.~J. Scheeres.
\newblock Solving {Relative} {Two}-{Point} {Boundary} {Value} {Problems}:
  {Spacecraft} {Formulation} {Flight} {Transfers} {Application}.
\newblock {\em Journal of Guidance, Control, and Dynamics}, 27(4):693--704,
  2004.

\bibitem{Haber_2017}
E.~Haber and L.~Ruthotto.
\newblock Stable architectures for deep neural networks.
\newblock {\em Inverse Problems}, 34(1):014004, dec 2017.

\bibitem{han_superlinearly_1976}
S.-P. Han.
\newblock Superlinearly convergent variable metric algorithms for general
  nonlinear programming problems.
\newblock {\em Mathematical Programming}, 11(1):263--282, Dec. 1976.

\bibitem{hinton_deep_2012-1}
G.~Hinton, L.~Deng, D.~Yu, G.~E. Dahl, A.~R. Mohamed, N.~Jaitly, A.~Senior,
  V.~Vanhoucke, P.~Nguyen, T.~N. Sainath, and B.~Kingsbury.
\newblock Deep {Neural} {Networks} for {Acoustic} {Modeling} in {Speech}
  {Recognition}: {The} {Shared} {Views} of {Four} {Research} {Groups}.
\newblock {\em IEEE Signal Processing Magazine}, 29(6):82--97, Nov. 2012.
\newblock Conference Name: IEEE Signal Processing Magazine.

\bibitem{holm_eulers_2009}
D.~D. Holm.
\newblock Euler's fluid equations: {Optimal} control vs optimization.
\newblock {\em Physics Letters A}, 373(47):4354--4359, Nov. 2009.

\bibitem{holm_poisson_1983}
D.~D. Holm and B.~A. Kupershmidt.
\newblock Poisson brackets and clebsch representations for
  magnetohydrodynamics, multifluid plasmas, and elasticity.
\newblock {\em Physica D: Nonlinear Phenomena}, 6(3):347--363, Apr. 1983.

\bibitem{HOLM19851}
D.~D. Holm, J.~E. Marsden, T.~Ratiu, and A.~Weinstein.
\newblock Nonlinear stability of fluid and plasma equilibria.
\newblock {\em Physics Reports}, 123(1):1 -- 116, 1985.

\bibitem{holm_eulerpoincare_1998}
D.~D. Holm, J.~E. Marsden, and T.~S. Ratiu.
\newblock The {Euler}-{Poincar\'{e}} {Equations} and {Semidirect} {Products}
  with {Applications} to {Continuum} {Theories}.
\newblock {\em Advances in Mathematics}, 137(1):1--81, July 1998.

\bibitem{metamorphosis}
D.~D. Holm, A.~Trouv\'{e}, and L.~Younes.
\newblock The euler poincar\'{e} theory of metamorphosis.
\newblock {\em Quarterly of Applied Mathematics}, 67(4):661--685, 2009.

\bibitem{horava_covariant_1991}
P.~Horava.
\newblock On a covariant {Hamilton}-{Jacobi} framework for the
  {Einstein}-{Maxwell} theory.
\newblock {\em Classical and Quantum Gravity}, 8(11):2069--2084, Nov. 1991.

\bibitem{junge_discrete_2005}
O.~Junge, J.~E. Marsden, and S.~Ober-Bl{\"o}baum.
\newblock Discrete mechanics and optimal control.
\newblock {\em IFAC Proceedings Volumes}, 38(1):538--543, 2005.

\bibitem{jurdjevic_1996}
V.~Jurdjevic.
\newblock {\em Geometric Control Theory}.
\newblock Cambridge Studies in Advanced Mathematics. Cambridge University
  Press, 1996.

\bibitem{jurdjevic_integrable_2011}
V.~Jurdjevic.
\newblock Integrable {Hamiltonian} systems on symmetric spaces: {Jacobi},
  {Kepler} and {Moser}.
\newblock {\em arXiv e-prints}, 1103:arXiv:1103.2818, Mar. 2011.

\bibitem{jurdjevic_2016}
V.~Jurdjevic.
\newblock {\em Optimal Control and Geometry: Integrable Systems}.
\newblock Cambridge Studies in Advanced Mathematics. Cambridge University
  Press, 2016.

\bibitem{kang_construction_1989}
F.~Kang, W.~Hua-mo, Q.~Meng-zhao, and W.~Dao-liu.
\newblock Construction of canonical difference schemes for hamiltonian
  formalism via generating functions.
\newblock {\em Journal of Computational Mathematics}, 7(1):71--96, 1989.

\bibitem{keller_corrected_1958}
J.~B. Keller.
\newblock Corrected bohr-sommerfeld quantum conditions for nonseparable
  systems.
\newblock {\em Annals of Physics}, 4(2):180--188, June 1958.

\bibitem{khesin_geometric_2018}
B.~Khesin, G.~Misiolek, and K.~Modin.
\newblock Geometric hydrodynamics via {Madelung} transform.
\newblock {\em Proceedings of the National Academy of Sciences},
  115(24):6165--6170, jun 2018.

\bibitem{khesin_geometric_2020}
B.~Khesin, G.~Misiolek, and K.~Modin.
\newblock Geometric {Hydrodynamics} of {Compressible} {Fluids}.
\newblock {\em arXiv:2001.01143 [math-ph]}, jan 2020.
\newblock arXiv: 2001.01143.

\bibitem{khesin2008geometry}
B.~Khesin and R.~Wendt.
\newblock {\em The geometry of infinite-dimensional groups}, volume~51.
\newblock Springer Science \& Business Media, 2008.

\bibitem{khesin_poisson-lie_1995}
B.~Khesin and I.~Zakharevich.
\newblock Poisson-{Lie} group of pseudodifferential symbols.
\newblock {\em Communications in Mathematical Physics}, 171(3):475--530, Aug.
  1995.

\bibitem{kobilarov_discrete_2011}
M.~B. Kobilarov and J.~E. Marsden.
\newblock Discrete {Geometric} {Optimal} {Control} on {Lie} {Groups}.
\newblock {\em IEEE Transactions on Robotics}, 27(4):641--655, aug 2011.

\bibitem{kolmogorov1954conservation}
A.~N. Kolmogorov.
\newblock On conservation of conditionally periodic motions for a small change
  in hamilton's function.
\newblock In {\em Dokl. Akad. Nauk SSSR}, volume~98, pages 527--530, 1954.

\bibitem{kouranbaeva_camassaholm_1999}
S.~Kouranbaeva.
\newblock The {Camassa}-{Holm} equation as a geodesic flow on the
  diffeomorphism group.
\newblock {\em Journal of Mathematical Physics}, 40(2):857--868, Jan. 1999.

\bibitem{kraft_converting_1985}
D.~Kraft.
\newblock On {Converting} {Optimal} {Control} {Problems} into {Nonlinear}
  {Programming} {Problems}.
\newblock In K.~Schittkowski, editor, {\em Computational {Mathematical}
  {Programming}}, {NATO} {ASI} {Series}, pages 261--280, Berlin, Heidelberg,
  1985. Springer.

\bibitem{krizhevsky_imagenet_2017}
A.~Krizhevsky, I.~Sutskever, and G.~E. Hinton.
\newblock {ImageNet} classification with deep convolutional neural networks.
\newblock {\em Communications of the ACM}, 60(6):84--90, May 2017.

\bibitem{lee_computational_2008}
T.~Lee, M.~Leok, and N.~H. McClamroch.
\newblock Computational {Geometric} {Optimal} {Control} of {Rigid} {Bodies}.
\newblock {\em Communications in Information \& Systems}, 8(4):445--472, Jan.
  2008.

\bibitem{lee_optimal_2008}
T.~Lee, M.~Leok, and N.~H. McClamroch.
\newblock Optimal {Attitude} {Control} of a {Rigid} {Body} {Using}
  {Geometrically} {Exact} {Computations} on {SO}(3).
\newblock {\em Journal of Dynamical and Control Systems}, 14(4):465--487, Oct.
  2008.

\bibitem{leok_discrete_2011}
M.~Leok and J.~Zhang.
\newblock Discrete {Hamiltonian} variational integrators.
\newblock {\em IMA Journal of Numerical Analysis}, 31(4):1497--1532, Oct. 2011.

\bibitem{leung2014deep}
M.~K. Leung, H.~Y. Xiong, L.~J. Lee, and B.~J. Frey.
\newblock Deep learning of the tissue-regulated splicing code.
\newblock {\em Bioinformatics}, 30(12):i121--i129, 2014.

\bibitem{li2017maximum}
Q.~Li, L.~Chen, C.~Tai, and E.~Weinan.
\newblock Maximum principle based algorithms for deep learning.
\newblock {\em The Journal of Machine Learning Research}, 18(1):5998--6026,
  2017.

\bibitem{pmlr-v80-li18b}
Q.~Li and S.~Hao.
\newblock An optimal control approach to deep learning and applications to
  discrete-weight neural networks.
\newblock In J.~Dy and A.~Krause, editors, {\em Proceedings of the 35th
  International Conference on Machine Learning}, volume~80 of {\em Proceedings
  of Machine Learning Research}, pages 2985--2994. PMLR, 10--15 Jul 2018.

\bibitem{ma_deep_2015}
J.~Ma, R.~P. Sheridan, A.~Liaw, G.~E. Dahl, and V.~Svetnik.
\newblock Deep neural nets as a method for quantitative structure--activity
  relationships.
\newblock {\em Journal of Chemical Information and Modeling}, 55(2):263--274,
  Feb. 2015.
\newblock Publisher: American Chemical Society.

\bibitem{madelung_quantentheorie_1927}
E.~Madelung.
\newblock Quantentheorie in hydrodynamischer {Form}.
\newblock {\em Zeitschrift f\"{u}r Physik}, 40(3):322--326, Mar. 1927.

\bibitem{magri1978simple}
F.~Magri.
\newblock A simple model of the integrable hamiltonian equation.
\newblock {\em Journal of Mathematical Physics}, 19(5):1156--1162, 1978.

\bibitem{marsden2007hamiltonian}
J.~Marsden, G.~Misiolek, J.~Ortega, M.~Perlmutter, and T.~Ratiu.
\newblock {\em Hamiltonian Reduction by Stages}.
\newblock Lecture Notes in Mathematics. Springer Berlin Heidelberg, 2007.

\bibitem{marsden_coadjoint_1983}
J.~Marsden and A.~Weinstein.
\newblock Coadjoint orbits, vortices, and {Clebsch} variables for
  incompressible fluids.
\newblock {\em Physica D: Nonlinear Phenomena}, 7(1):305--323, may 1983.

\bibitem{marsden1998multisymplectic}
J.~E. Marsden, G.~W. Patrick, and S.~Shkoller.
\newblock Multisymplectic {Geometry}, {Variational} {Integrators}, and
  {Nonlinear} {PDEs}.
\newblock {\em Communications in Mathematical Physics}, 199(2):351--395, Dec.
  1998.

\bibitem{Marsden_1999}
J.~E. Marsden, S.~Pekarsky, and S.~Shkoller.
\newblock Discrete euler-poincar{\'{e}} and lie-poisson equations.
\newblock {\em Nonlinearity}, 12(6):1647--1662, oct 1999.

\bibitem{marsden1983hamiltonian}
J.~E. Marsden, T.~Ratiu, R.~Schmid, R.~Spencer, and A.~J. Weinstein.
\newblock Hamiltonian systems with symmetry, coadjoint orbits and plasma
  physics.
\newblock {\em Atti della Accademia delle scienze di Torino}, 117(1):289--340,
  1983.

\bibitem{marsden_semidirect_1984}
J.~E. Marsden, T.~Ratiu, and A.~Weinstein.
\newblock Semidirect {Products} and {Reduction} in {Mechanics}.
\newblock {\em Transactions of the American Mathematical Society},
  281(1):147--177, 1984.

\bibitem{marsden2013introduction}
J.~E. Marsden and T.~S. Ratiu.
\newblock {\em Introduction to mechanics and symmetry: a basic exposition of
  classical mechanical systems}, volume~17.
\newblock Springer Science \& Business Media, 2013.

\bibitem{marsden2001discrete}
J.~E. Marsden and M.~West.
\newblock Discrete mechanics and variational integrators.
\newblock {\em Acta Numerica}, 10:357--514, 2001.

\bibitem{mclachlan_explicit_1993}
R.~I. McLachlan.
\newblock Explicit {Lie}-{Poisson} integration and the {Euler} equations.
\newblock {\em Physical Review Letters}, 71(19):3043--3046, Nov. 1993.
\newblock Publisher: American Physical Society.

\bibitem{mclean_covariant_2000}
M.~McLean and L.~K. Norris.
\newblock Covariant field theory on frame bundles of fibered manifolds.
\newblock {\em Journal of Mathematical Physics}, 41(10):6808--6823, Sept. 2000.

\bibitem{misio_lek_shallow_1998}
G.~Misio{\l}ek.
\newblock A shallow water equation as a geodesic flow on the {Bott}-{Virasoro}
  group.
\newblock {\em Journal of Geometry and Physics}, 24(3):203--208, Feb. 1998.

\bibitem{monaghan1992smoothed}
J.~J. Monaghan.
\newblock Smoothed particle hydrodynamics.
\newblock {\em Annual review of astronomy and astrophysics}, 30(1):543--574,
  1992.

\bibitem{moser_invariant_1962}
J.~Moser.
\newblock On invariant curves of area-preserving mappings of an annulus.
\newblock {\em Nachrichten der Akademie der Wissenschaften in G\"{o}ttingen.
  II. Mathematisch-Physikalische Klasse}, 1962:1--20, 1962.
\newblock Publisher: Vandenhoeck \& Ruprecht, G\"{o}ttingen.

\bibitem{mullen2009energy}
P.~Mullen, K.~Crane, D.~Pavlov, Y.~Tong, and M.~Desbrun.
\newblock Energy-preserving integrators for fluid animation.
\newblock In {\em ACM Transactions on Graphics (TOG)}, volume~28, page~38. ACM,
  2009.

\bibitem{Ohsawa2011DiscreteHT}
T.~Ohsawa, A.~Bloch, and M.~Leok.
\newblock Discrete hamilton-jacobi theory.
\newblock {\em SIAM J. Control. Optim.}, 49:1829--1856, 2011.

\bibitem{park2003solutions}
C.~Park and D.~J. Scheeres.
\newblock Solutions of the optimal feedback control problem using hamiltonian
  dynamics and generating functions.
\newblock In {\em 42nd IEEE International Conference on Decision and Control
  (IEEE Cat. No. 03CH37475)}, volume~2, pages 1222--1227. IEEE, 2003.

\bibitem{pavlov2011structure}
D.~Pavlov, P.~Mullen, Y.~Tong, E.~Kanso, J.~E. Marsden, and M.~Desbrun.
\newblock Structure-preserving discretization of incompressible fluids.
\newblock {\em Physica D: Nonlinear Phenomena}, 240(6):443--458, 2011.

\bibitem{persio_deep_2021}
L.~D. Persio and M.~Garbelli.
\newblock Deep {Learning} and {Mean}-{Field} {Games}: {A} {Stochastic}
  {Optimal} {Control} {Perspective}.
\newblock {\em Symmetry}, 13(1):14, Jan. 2021.

\bibitem{doi:10.1002/zamm.19630431023}
L.~S. Pontryagin, V.~G. Boltyanskii, R.~V. Gamkrelidze, and E.~F. Mishchenko.
\newblock {The mathematical theory of optimal processes}.
\newblock 1962.

\bibitem{rezende2015variational}
D.~Rezende and S.~Mohamed.
\newblock Variational inference with normalizing flows.
\newblock In {\em International Conference on Machine Learning}, pages
  1530--1538. PMLR, 2015.

\bibitem{ruthotto_deep_2019}
L.~Ruthotto and E.~Haber.
\newblock Deep {Neural} {Networks} {Motivated} by {Partial} {Differential}
  {Equations}.
\newblock {\em Journal of Mathematical Imaging and Vision}, sep 2019.

\bibitem{scheeres2003solving}
D.~J. Scheeres, C.~Park, and V.~M. Guibout.
\newblock Solving optimal control problems with generating functions.
\newblock {\em Advances in the Astronautical Sciences}, 116:1185--1205, 2003.

\bibitem{seide_conversational_nodate}
F.~Seide, G.~Li, and D.~Yu.
\newblock Conversational {Speech} {Transcription} {Using} {Context}-{Dependent}
  {Deep} {Neural} {Networks}.
\newblock page~4.

\bibitem{SIMO199263}
J.~Simo, N.~Tarnow, and K.~Wong.
\newblock Exact energy-momentum conserving algorithms and symplectic schemes
  for nonlinear dynamics.
\newblock {\em Computer Methods in Applied Mechanics and Engineering},
  100(1):63--116, 1992.

\bibitem{Simo1991StabilityOR}
J.~C. Simo, D.~Lewis, and J.~E. Marsden.
\newblock Stability of relative equilibria. part i: The reduced energy-momentum
  method.
\newblock {\em Archive for Rational Mechanics and Analysis}, 115:15--59, 1991.

\bibitem{sun_neupde_2019}
Y.~Sun, L.~Zhang, and H.~Schaeffer.
\newblock {NeuPDE}: {Neural} {Network} {Based} {Ordinary} and {Partial}
  {Differential} {Equations} for {Modeling} {Time}-{Dependent} {Data}.
\newblock {\em arXiv:1908.03190 [cs, stat]}, Aug. 2019.
\newblock arXiv: 1908.03190.

\bibitem{sussmann_geometry_1999}
H.~J. Sussmann.
\newblock Geometry and {Optimal} {Control}.
\newblock In J.~Baillieul and J.~C. Willems, editors, {\em Mathematical
  {Control} {Theory}}, pages 140--198. Springer, New York, NY, 1999.

\bibitem{10.5555/2969033.2969173}
I.~Sutskever, O.~Vinyals, and Q.~V. Le.
\newblock Sequence to sequence learning with neural networks.
\newblock In {\em Proceedings of the 27th International Conference on Neural
  Information Processing Systems - Volume 2}, NIPS'14, pages 3104--3112,
  Cambridge, MA, USA, 2014. MIT Press.

\bibitem{szegedy_going_2014}
C.~Szegedy, W.~Liu, Y.~Jia, P.~Sermanet, S.~Reed, D.~Anguelov, D.~Erhan,
  V.~Vanhoucke, and A.~Rabinovich.
\newblock Going {Deeper} with {Convolutions}.
\newblock {\em arXiv:1409.4842 [cs]}, Sept. 2014.
\newblock arXiv: 1409.4842.

\bibitem{vankerschaver_generating_2013}
J.~Vankerschaver, C.~Liao, and M.~Leok.
\newblock Generating functionals and {Lagrangian} partial differential
  equations.
\newblock {\em Journal of Mathematical Physics}, 54(8):082901, Aug. 2013.

\bibitem{von_rieth_hamilton-jacobi_1984}
J.~von Rieth.
\newblock The {Hamilton}-{Jacobi} theory of {De} {Donder} and {Weyl} applied to
  some relativistic field theories.
\newblock {\em JMP}, 25(4):1102--1115, Apr. 1984.

\bibitem{weinan2017proposal}
E.~Weinan.
\newblock A proposal on machine learning via dynamical systems.
\newblock {\em Communications in Mathematics and Statistics}, 5(1):1--11, 2017.

\bibitem{weinan2019mean}
E.~Weinan, J.~Han, and Q.~Li.
\newblock A mean-field optimal control formulation of deep learning.
\newblock {\em Research in the Mathematical Sciences}, 6(1):10, 2019.

\bibitem{xiong2015human}
H.~Y. Xiong, B.~Alipanahi, L.~J. Lee, H.~Bretschneider, D.~Merico, R.~K. Yuen,
  Y.~Hua, S.~Gueroussov, H.~S. Najafabadi, T.~R. Hughes, et~al.
\newblock The human splicing code reveals new insights into the genetic
  determinants of disease.
\newblock {\em Science}, 347(6218), 2015.

\bibitem{yoshida_construction_1990}
H.~Yoshida.
\newblock Construction of higher order symplectic integrators.
\newblock {\em Physics Letters A}, 150(5):262--268, Nov. 1990.

\bibitem{zeitlin_finite-mode_1991}
V.~Zeitlin.
\newblock Finite-mode analogs of {2D} ideal hydrodynamics: {Coadjoint} orbits
  and local canonical structure.
\newblock {\em Physica D: Nonlinear Phenomena}, 49(3):353--362, Apr. 1991.

\bibitem{zhong1988lie}
G.~Zhong and J.~E. Marsden.
\newblock {Lie-Poisson Hamilton-Jacobi theory and Lie-Poisson integrators}.
\newblock {\em Physics Letters A}, 133(3):134--139, 1988.

\end{thebibliography}


\bibliographystyle{abbrv}

\end{document}